\documentclass[10pt,reqno,final]{article}

\usepackage[T1]{fontenc}
\usepackage[utf8]{inputenc}
\usepackage[english]{babel}
\usepackage{geometry}
\usepackage{amsmath,amssymb,amsthm,mathtools,mathrsfs,eucal}
\mathtoolsset{showonlyrefs}

\usepackage{graphicx}
\usepackage{subcaption}
\usepackage{tikz}

\usepackage{booktabs,multirow,rotating,varioref,footmisc}
\usepackage{enumerate}
\usepackage[colorlinks=true,linkcolor=black,citecolor=black,urlcolor=blue]{hyperref}
\usepackage{braket}
\input xy
\xyoption{all}

\theoremstyle{plain}
\newtheorem{mainthm}{\textsc{Theorem}}
\newtheorem{thm}{Theorem}[section]
\newtheorem{cor}[thm]{Corollary}	 
\newtheorem{maincor}{\textsc{Corollary}}
\newtheorem{lem}[thm]{Lemma}		
\newtheorem{prop}[thm]{Proposition}

\theoremstyle{definition}
\newtheorem{defn}[thm]{Definition}	
\newtheorem{ex}[thm]{Example}

\theoremstyle{remark}
\newtheorem{rem}[thm]{Remark}
\newtheorem{note}[thm]{Notation}

\numberwithin{equation}{section}

\newcommand{\R}{\mathbb{R}}
\newcommand{\Com}{\mathbb{C}}
\newcommand{\Q}{\mathbb{Q}}
\newcommand{\Z}{\mathbb{Z}}	
\newcommand{\N}{\mathbb{N}}	
\renewcommand{\P}{\mathscr{P}}
\newcommand{\Lagr}{L}

\newcommand{\norm}[1]{\left\| #1 \right\|}
\newcommand{\cfsa}{\mathscr{CF}^{sa}}
\newcommand{\im}{\mathrm{rge}\,}
\newcommand{\Mat}{\mathrm{Mat}\,}
\newcommand{\Sp}{\mathrm{Sp}\,}
\newcommand{\Sym}{\mathrm{Sym}\,}
\newcommand{\Lin}{\mathscr{L}}

\newcommand{\Id}{I}

\newcommand{\Sf}[1]{\mathrm{Sf\,}(#1)}

\newcommand{\trasp}[1]{{#1}^\mathsf{T}}
\DeclareMathOperator{\sgn}{sgn}
\DeclareMathOperator{\ind}{ind}
\DeclareMathOperator{\iCLM}{\iota^{\scriptscriptstyle{\mathrm{CLM}}}}
\DeclareMathOperator{\iMor}{m^{--}}
\DeclareMathOperator{\coiMor}{m^{+}}

\DeclareMathOperator{\igeo}{\iota}

\DeclareMathOperator{\dom}{dom}
\DeclareMathOperator{\codim}{codim}
\DeclareMathOperator{\sech}{sech}

\renewcommand{\=}{\coloneqq}

\begin{document}

\title{Morse index theorem for heteroclinic, homoclinic and halfclinic orbits of Lagrangian systems}
\author{%
Xijun Hu\thanks{Partially supported by  the Taishan Scholars Climbing Program of Shandong (TSPD20240802).} \and
Alessandro Portaluri\thanks{Partially supported by Progetto di Ricerca GNAMPA--INdAM, codice CUP-E55F22000270001
``Dinamica simbolica e soluzioni periodiche per problemi singolari della Meccanica Celeste'' and PRIN 2022–2025: ``Stability in Hamiltonian dynamics and beyond''} \and
Li Wu\thanks{Partially supported by NSFC (No.\ 12171281)} \and
Qin Xing\thanks{Partially supported by NSFC (No.\ 12201278) and by the Natural Science Foundation of Shandong Province (No.\ ZR2022QA076)}}
\maketitle

\begin{abstract}
We prove a new, more general version of the Morse index theorem for heteroclinic, homoclinic, and halfclinic solutions of general Lagrangian systems. In the final section we compute the Morse index for explicit heteroclinic and halfclinic solutions in classical mechanical models such as the mathematical pendulum, the Nagumo equation, and a four–dimensional competition–diffusion system.

\end{abstract}


\section{Introduction, description of the problem and main results}\label{sec:introduction}

Morse index theory for Lagrangian systems relates the Morse index of a critical point of a Legendre–convex variational problem to the symplectic oscillatory behavior of the associated linearized equation at that point. The subject goes back to M.~Morse, who first expressed the index of a geodesic (as a critical point of the geodesic action) in terms of the total number of conjugate points, counted with multiplicity. This theory has since been extended by Edwards, Simons, and Smale to higher–order systems, minimal surfaces, and certain PDEs.

Classically, most results concern Hamiltonian orbits on compact time intervals. A breakthrough came with Chen and Hu \cite{CH07}, who initiated the study of unbounded trajectories. Their work stimulated many developments: Chardard, Bridges, and Dias extended the Maslov index framework to solitary waves and multi–pulse homoclinic orbits \cite{CB15,CDB09a,CDB09b,CDB11}. In 2008, Pejsachowicz \cite{Pej08} showed that homoclinic trajectories of nonautonomous vector fields on the circle bifurcate from a stationary solution when the asymptotic stable bundles $E^s(\pm\infty)$ have different twists. His proof reveals a deep link between the topology of these asymptotic bundles and the birth of homoclinic solutions, and forms an essential building block for index theory in unbounded settings.

Index–theoretic results for heteroclinic, homoclinic and halfclinic (h–clinic) 
trajectories appeared only later. Waterstraat obtained a spectral–flow formula 
for homoclinics \cite{Waa15}, and Hu–Portaluri developed an index theory for 
h–clinic solutions \cite{HP17}. Earlier, Jones–Marangell \cite{JM12} studied 
the stability of a travelling wave, in effect using a spectral–flow approach to 
a heteroclinic orbit.

Subsequent works 
\cite{HLS18,HS20,How21,How23,How25,Waa21} further explored the link between the 
Morse and Maslov indices and the oscillatory behaviour of solutions on the 
half–line. Below we briefly compare our results with these contributions.

In \cite{HLS17,HLS18}, the authors consider the eigenvalue problem
\[
Hy := -y^{\prime\prime} + V(x)y = \lambda y, \quad \text{dom}(H) = H^{2}(\mathbb{R}),
\]
where $\lambda\in\R$ and $V\in \mathscr C^0(\R;\R^{n\times n})$ is a symmetric matrix potential, under suitable integral constraints. Their main theorem \cite[Theorem~1.2]{HLS18} gives
\[
\operatorname{Mor}(H)
= -\operatorname{Mas}\bigl(E^u(\tau),E^s(+\infty);\tau\in\R\bigr).
\]
In the present paper, under significantly milder hypotheses (notably without integral conditions), we obtain the same conclusion in Corollary~\ref{thm:morse index=maslov index}, thereby covering a wider class of systems.

Later, Howard and Sukhtayev \cite{HS20} studied the Sturm–Liouville operator
\[
\mathcal{L}\phi = Q(x)^{-1} \Bigl( -\bigl(P(x)\phi^{\prime}\bigr)^{\prime} + V(x)\phi \Bigr),
\quad \phi(x) \in W^{2,2}([0,\infty),\Com^{n}),
\]
with Lagrangian boundary conditions and asymptotic assumptions on $P,V,Q$. Their main result \cite[Theorem~1.1]{HS20} is the identity
\begin{equation}\label{eq:how-suk} 
\operatorname{Mor}(\mathcal{L}) = \operatorname{Mas}(\Lambda_0,E^{s}(\tau); \tau \in [0,\infty]) - \operatorname{Mas}(\Lambda_0,E_{\lambda}^{s}(+\infty); \lambda \in [-\lambda_{\infty},0])
\end{equation}
for sufficiently large $\lambda_\infty$. In Theorem~\ref{thm:morse index formula for halfclinic}, under weaker assumptions, we obtain a stronger statement by explicitly computing the second term on the right-hand side of \eqref{eq:how-suk} via the H\"ormander and triple indices.

In \cite{HP17}, the authors constructed an index theory for h–clinic motions of general Hamiltonian systems, and in \cite{BHPT19} they provided an ad hoc extension to certain asymptotic motions in weakly singular Lagrangian systems (including the gravitational $n$–body problem).

Starting from the spectral flow formula in \cite{HP17}, we construct here an index theory for h–clinic motions in the Lagrangian setting. The theory identifies the {\sc Morse index} of an h–clinic solution with a {\sc geometric index} defined via a Maslov–type index, up to an explicit correction term.

Finally, we apply our main results to compute the Morse index of specific heteroclinic and halfclinic solutions in the following classical models:
\begin{itemize}
\item the mathematical pendulum;
\item the Nagumo reaction equation for impulse propagation along a nerve fibre;
\item a reaction–diffusion system in $\R^4$.
\end{itemize}

\subsection{Description of the problem and main results}\label{subsec:main-results}

Let $T\R^n\cong \R^n \times \R^n$ denote the tangent bundle of $\R^n$; its elements are written $(q,v)$ with $q\in\R^n$ and $v\in T_q\R^n\cong\R^n$. Let
\[
L:\R\times T\R^n\to\R
\]
be a smooth nonautonomous Lagrangian satisfying the Legendre convexity condition
\begin{itemize}
\item[{\bf(L1)}] $L$ is $\mathscr C^2$–convex on the fibres of $T\R^n$, i.e.
\[
\norm{D^2_{vv} L(t,q,v)} \ge \ell_0 \Id>0 \qquad \forall (t,q,v)\in\R\times T\R^n.
\]
\end{itemize} 
We fix two rest points $u^-,u^+\in\R^n$ of the Lagrangian vector field $\nabla L$, i.e.
\[
\nabla L(t,u^\pm,0)=0 \quad \text{for all } t\in\R.
\]

\begin{defn}\label{def:hetero}
A \emph{heteroclinic orbit} $u$ \emph{from $u^-$ to $u^+$} is a $\mathscr C^2$–solution of
\begin{equation}\label{eq:u.l.e.}
\begin{cases}
\dfrac{d}{dt}\partial_v L\bigl(t,u(t),\dot u(t)\bigr)
 = \partial_q L\bigl(t,u(t),\dot u(t)\bigr), & t\in\R,\\[4pt]
\displaystyle\lim_{t\to -\infty} u(t)=u^-,\quad
\lim_{t\to +\infty} u(t)=u^+.
\end{cases}
\end{equation}	
If $u^-=u^+$, we call $u$ a \emph{homoclinic} orbit.
\end{defn}

\begin{defn}\label{def:hclinics}
Let $L_0\in \Lagr(n)$ be a Lagrangian subspace of $(\R^{2n},\omega)$.
A \emph{future halfclinic solution} $u$ starting at $L_0$ is a solution of
\begin{equation}\label{eq:half}
\begin{cases}
\dfrac{d}{dt}\partial_v L\bigl(t,u(t),\dot u(t)\bigr)
 = \partial_q L\bigl(t,u(t),\dot u(t)\bigr), & t\in\R^+,\\[4pt]
\trasp{\bigl(\partial_v L(0,u(0),\dot u(0)),u(0)\bigr)}\in L_0,\\[4pt]
\displaystyle\lim_{t\to +\infty} u(t)=u^+.
\end{cases}
\end{equation}	
A \emph{past halfclinic solution} $u$ starting at $L_0$ is defined analogously on $\R^-$, with
$\lim_{t\to -\infty}u(t)=u^-$.
\end{defn}

\begin{note}
Any solution of \eqref{eq:u.l.e.} or \eqref{eq:half} will be called an \emph{h–clinic orbit}. We write
$I:=\R$, $\R^+$, or $\R^-$ according to the case.
\end{note}

Linearizing the Euler–Lagrange equation along an h–clinic orbit $u$ and setting
\begin{align}\label{def:PQR}
P(t)\=\partial_{vv} L\bigl(t,u(t),\dot u(t)\bigr),\quad 
Q(t)\=\partial_{uv} L\bigl(t,u(t),\dot u(t)\bigr),\quad 
R(t)\=\partial_{uu} L\bigl(t,u(t),\dot u(t)\bigr),
\end{align}
we obtain the variational equation and the associated Sturm–Liouville operator
\begin{equation}\label{eq:Sturm-Liouville}
(\mathscr A\, w)(t)
:= -\dfrac{d}{dt}\bigl(P(t)\dot w(t)+ Q(t)w(t)\bigr)
   + \trasp{Q(t)}\,\dot w(t) + R(t)\, w(t),
\qquad t\in I.
\end{equation}
In the halfclinic case, the boundary condition at $t=0$ becomes
\[
\bigl(P(0)\dot w(0)+Q(0)w(0),\,w(0)\bigr)^{T}\in L_0.
\]
The functions $P$ and $R$ are symmetric matrix paths. We assume:

\begin{itemize}
\item[{\bf(L2)}]
$P(t),Q(t),R(t)$ converge as $t\to\pm\infty$ to matrices $P_\pm,Q_\pm,R_\pm$, respectively, and there are constants $C_1,C_2,C_3>0$ such that
\[
\norm{P(t)}\ge C_1,\quad
\norm{Q(t)}\le C_2,\quad
\norm{R(t)}\le C_3 \qquad \forall t\in I.
\]
\end{itemize}

Let
\[
z(t)=\begin{pmatrix} P(t)\dot{w}(t)+Q(t)w(t)\\ w(t)\end{pmatrix}.
\]
Then \eqref{eq:Sturm-Liouville} corresponds to the Hamiltonian system
\begin{equation}\label{eq:Ham-sys-family}
\dot z = J B(t)z,\qquad
B(t)\=\begin{pmatrix}
P^{-1}(t) & -P^{-1}(t)Q(t)\\
-\trasp{Q}(t)P^{-1}(t) & \trasp{Q}(t)P^{-1}(t)Q(t)-R(t)
\end{pmatrix},
\end{equation}
and we impose:

\begin{itemize}
\item[\bf(H1)] The limit matrices $JB(-\infty)$ and $JB(+\infty)$ are hyperbolic (their spectrum avoids the imaginary axis).
\item[\bf(H2)] The matrices
\(
\begin{pmatrix}
P_-&Q_-\\Q_-^T&R_-
\end{pmatrix}
\)
and
\(
\begin{pmatrix}
P_+&Q_+\\Q_+^T&R_+
\end{pmatrix}
\)
are both positive definite.
\end{itemize} 

Set
\[
E:=W^{2,2}(\R,\R^{n}),\qquad
E^\pm_{L_0}:=\Bigl\{w\in W^{2,2}(\R^\pm,\R^{n})\ \Big|\ 
\bigl(P(0)\dot w(0)+Q(0)w(0),\,w(0)\bigr)^T\in L_0\Bigr\}.
\]
Define
\[
\mathcal A := \mathscr A\big|_{E},
\qquad
\mathcal A_{L_0}^{\pm} := \mathscr A\big|_{E_{L_0}^{\pm}} .
\]
For the selfadjoint operator $\mathcal A$ arising from the second variation at a critical point $u$, we denote by
\[
\iMor(\mathcal A)
\]
its \emph{Morse index}, namely the dimension of the maximal subspace of $E$ on which the quadratic form associated with the second variation is negative definite. Equivalently, $\iMor(\mathcal A)$ equals the total multiplicity of the negative eigenvalues of $\mathcal A$. For h–clinic solutions we use the shorthand
\begin{itemize}
\item $u$ heteroclinic:
\[
\iMor(u):=\iMor(\mathcal A);
\]
\item $u$ future half–clinic with asymptotics $L_0$:
\[
\iMor(u,L_0,+):=\iMor(\mathcal A_{L_0}^+);
\]
\item $u$ past half–clinic with asymptotics $L_0$:
\[
\iMor(u,L_0,-):=\iMor(\mathcal A_{L_0}^-).
\]
\end{itemize}
Let $\gamma_\tau$ be the fundamental solution of
\begin{equation}\label{eq:Ham-sys-bvp-het-hom-intro-single}
\begin{cases}
\dot \gamma_\tau(t)=JB(t)\,\gamma_\tau(t), & t\in\R,\\
\gamma_\tau(\tau)=\Id,
\end{cases}
\end{equation}
and denote by $E^{s}(\tau)$ and $E^{u}(\tau)$ the associated stable and unstable Lagrangian subspaces. The asymptotic stable/unstable subspaces $E^{s}(+\infty)$ and $E^{u}(-\infty)$ are defined as the negative and positive spectral subspaces of $JB(+\infty)$ and $JB(-\infty)$, respectively. Under (H1),
\[
\lim_{\tau\to +\infty}E^{s}(\tau)=E^{s}(+\infty),
\qquad
\lim_{\tau\to -\infty}E^{u}(\tau)=E^{u}(-\infty)
\]
in the gap topology on the Lagrangian Grassmannian (see \cite{AM03}).

\begin{defn}[{\cite{HP17}}]\label{def:h-clinic index}
The \emph{geometrical index} of an h–clinic orbit $u$ is defined as follows:
\begin{itemize}
\item for a heteroclinic solution $u$ of \eqref{eq:u.l.e.},
\[
\igeo(u):=-\iCLM\bigl(E^{s}(\tau),E^{u}(-\tau);\tau\in\R^+\bigr);
\]
\item for a future halfclinic solution $u$ of \eqref{eq:half},
\[
\igeo^+_{L_0}(u):=-\iCLM\bigl(E^{s}(\tau),L_0;\tau\in\R^+\bigr);
\]
\item for a past halfclinic solution $u$ of \eqref{eq:half},
\[
\igeo^-_{L_0}(u):=-\iCLM\bigl(L_0,E^{u}(-\tau);\tau\in\R^+\bigr),
\]
\end{itemize}
where $\iCLM$ denotes the Cappell–Lee–Miller Maslov index (see Appendix~\ref{sec:maslov}).
\end{defn}

For three Lagrangian subspaces $L_1,L_2,L_3$ we denote by
\[
\iota(L_1,L_2,L_3)
\]
the \emph{triple index} (see Appendix~\ref{lem:defi:triple index} for the definition and basic properties). Let $L_D:=\R^n\times\{0\}$ be the Dirichlet Lagrangian. We can now state the main result for heteroclinic solutions.

\begin{mainthm}\label{thm:morse-index-formula}
Let $u$ be a heteroclinic solution and assume {\rm(L1)}, {\rm(L2)}, and {\rm(H1)}. Then
\begin{equation}\label{eq:Morse Maslov}
\iMor(u)=\igeo(u)+\iota\bigl(E^{u}(-\infty),E^{s}(+\infty);L_D\bigr).
\end{equation}
\end{mainthm}

\begin{maincor}\label{thm:morse index=maslov index}
Let $u$ be a heteroclinic solution and assume {\rm(L1)}, {\rm(L2)}, {\rm(H1)}, {\rm(H2)}. Then
\begin{equation}\label{eq:Morse Maslov-corollary}
\iMor(\mathcal A)=\igeo(u).
\end{equation}
\end{maincor}

Our next theorem is the analogue for past and future halfclinic solutions and general Lagrangian boundary conditions.

\begin{mainthm}\label{thm:morse index formula for halfclinic}
Let $u$ be a future or past halfclinic orbit. Assume {\rm(L1)}, {\rm(L2)}, and {\rm(H1)}. Then:
\begin{itemize}
\item[] (future)
\begin{equation}\label{eq:Morse Maslov+}
\iMor(u,L_0,+)=\igeo^+_{L_0}(u)+\iota\bigl(L_D,L_0;E^{s}(+\infty)\bigr);
\end{equation}
\item[] (past)
\begin{equation}\label{eq:Morse Maslov -}
\iMor(u,L_0,-)=\igeo^-_{L_0}(u)+\iota\bigl(E^{u}(-\infty),L_0;L_D\bigr).
\end{equation}
\end{itemize}
\end{mainthm}

As a direct consequence we obtain a comparison of Morse indices when the Lagrangian boundary condition is replaced by the Dirichlet one.

\begin{maincor}\label{cor:morse index between general boundary and Dirichlet boundary}
Let $u$ be a future or past halfclinic orbit and assume {\rm(L1)}, {\rm(L2)}, {\rm(H1)}, {\rm(H2)}. Then:
\begin{itemize}
\item[] (future)
\begin{equation}\label{eq:+morse index between general boundary and Dirichlet boundary}
\iMor(u,L_0,+)-\iMor(u,L_D,+)=\iota\bigl(L_D,L_0;E^s(0)\bigr);
\end{equation}
\item[] (past)
\begin{equation}\label{eq:-morse index between general boundary and Dirichlet boundary}
\iMor(u,L_0,-)-\iMor(u,L_D,-)=\iota\bigl(E^u(0),L_0;L_D\bigr).
\end{equation}
\end{itemize}
\end{maincor}

\begin{ex}[The scalar case]\label{ex:scalar}
Consider the scalar Sturm–Liouville operator $\mathscr A_\lambda$ defined by \eqref{eq:Sturm-Liouville} with $n=1$ and $R$ replaced by $R_\lambda:=R+\lambda$. From Section~\ref{sec:fredholm-sturm} we know that $\mathscr A_\lambda$ is Fredholm if and only if $JB_\lambda(\pm\infty)$ are hyperbolic, which in this scalar case is equivalent to $R_\pm>0$.

Since $P,Q,R$ are real scalar functions, evaluating at $+\infty$ gives
\[
\det\bigl(\mu-JB_\lambda(+\infty)\bigr)
=\mu^2-P_+^{-1}(R_++\lambda).
\]
Thus the eigenvalues of $JB_\lambda(+\infty)$ are $\pm\sqrt{P_+^{-1}(R_++\lambda)}$, and
\[
JB_\lambda(+\infty)
\begin{pmatrix}
Q_+\pm\sqrt{P_+(R_++\lambda)}\\[2pt] 1
\end{pmatrix}
=\pm \sqrt{P_+^{-1}(R_++\lambda)}
\begin{pmatrix}
Q_+\pm\sqrt{P_+(R_++\lambda)}\\[2pt] 1
\end{pmatrix}.
\]
Hence
\[
E^s_\lambda(+\infty)
=V^{-}\bigl(JB_\lambda(+\infty)\bigr)
=\operatorname{span}\left\{
\begin{pmatrix}
Q_+-\sqrt{P_+(R_++\lambda)}\\[2pt]1
\end{pmatrix}
\right\},
\]
and similarly
\[
E^u_\lambda(-\infty)
=\operatorname{span}\left\{
\begin{pmatrix}
Q_-+\sqrt{P_-(R_-+\lambda)}\\[2pt]1
\end{pmatrix}
\right\}.
\]

\begin{figure}[htp]
\begin{subfigure}[b]{0.45\textwidth}
\centering
\begin{tikzpicture}
\draw[->] (-2,0)--(2,0) ;
\node[above] at (2,0) {$x$};
\draw[->] (0,-2)--(0,2);
\node[right] at (0,2) {$y$};

\draw[green,domain=-1.8:1.8] plot(\x,{-0.3*\x});
\node[above] at (-1.8,0.54) {\small $L^+_\lambda$};
\draw[green,domain=-0.9:0.9] plot(\x,{-2*\x});
\node[above] at (-0.9,1.8) {$L^+_0$};
\draw[<-,domain=-0.99:-0.45] plot(\x,{sqrt(1-\x*\x)});

\draw[blue,domain=-1.8:1.8] plot(\x,{0.3*\x});
\node[above] at (1.8,0.54) {$L^-_\lambda$};
\draw[blue,domain=-0.9:0.9] plot(\x,{2*\x});
\node[above] at (0.9,1.8) {$L^-_0$};
\draw[->,domain=0.45:0.99] plot(\x,{sqrt(1-\x*\x)});
\end{tikzpicture}
\caption{$L_\lambda^+$ approaches the $x$–axis counterclockwise and $L_\lambda^-$ clockwise; no coincidence times on $[0,\widehat \lambda]$.}
\label{Fig.1}
\end{subfigure}
\quad
\begin{subfigure}[b]{0.45\textwidth}
\centering
\begin{tikzpicture}
\draw[->] (-2,0)--(2,0) ;
\node[above] at (2,0) {$x$};
\draw[->] (0,-2)--(0,2);
\node[right] at (0,2) {$y$};

\draw[green,domain=-1.8:1.8] plot(\x,{-0.3*\x});
\node[above] at (-1.8,0.54) {$L^+_\lambda$};
\draw[green,domain=-0.9:0.9] plot(\x,{-2*\x});
\node[above] at (-0.9,1.8) {$L^+_0$};
\draw[->,domain=-0.96:-0.3] plot(\x,{sqrt(1-\x*\x)});

\draw[blue,domain=-1.8:1.8] plot(\x,{0.3*\x});
\node[above] at (1.8,0.54) {$L^-_\lambda$};
\draw[blue,domain=-0.9:0.9] plot(\x,{2*\x});
\node[above] at (0.9,1.8) {$L^-_0$};
\draw[<-,domain=0.3:0.96] plot(\x,{sqrt(1-\x*\x)});
\end{tikzpicture}
\caption{$L_\lambda^+$ approaches the $x$–axis clockwise and $L_\lambda^-$ counterclockwise; one coincidence time on $[0,\widehat\lambda]$.}
\label{Fig.2}
\end{subfigure}
\end{figure}

Thus $E^s_\lambda(+\infty)$ is the line $L_\lambda^+$ through the origin with slope (relative to the $y$–axis) $Q_+-\sqrt{P_+(R_++\lambda)}$, and $E^u_\lambda(-\infty)$ is the line $L_\lambda^-$ with slope $Q_-+\sqrt{P_-(R_-+\lambda)}$. As $\lambda\to+\infty$, $L_\lambda^+$ approaches the $x$–axis counterclockwise and $L_\lambda^-$ clockwise. Hence
\[
\iCLM\bigl(E^s_\lambda(+\infty),E^u_\lambda(-\infty);\lambda\in[0,\widehat{\lambda}]\bigr)
\]
is the number of coincidence times (with multiplicity) of $L_\lambda^+$ and $L_\lambda^-$ for $\lambda\in[0,\widehat\lambda]$.

We distinguish two cases:

\begin{itemize}
\item[][{\bf Case 1}] (Figure~\ref{Fig.1}). If (H2) holds, then $\sqrt{P_\pm R_\pm}\mp Q_\pm >0$, so $L^+_0$ lies in the left, and $L^-_0$ in the right half–plane bounded by the $y$–axis. Hence there are no coincidence times and
\[
\iCLM\bigl(E^s_\lambda(+\infty),E^u_\lambda(-\infty);\lambda\in[0,\widehat{\lambda}]\bigr)=0,
\]
so by the Morse index formula we obtain $\iMor(u)=\igeo(u)$.

\item[][{\bf Case 2}] (Figure~\ref{Fig.2}). If $\sqrt{P_\pm R_\pm}\mp Q_\pm <0$, $Q_+>0$, and $Q_-<0$, then $L^+_0$ lies in the right and $L^-_0$ in the left half–plane. The lines $L^+_\lambda$ and $L^-_\lambda$ intersect exactly once as $\lambda\to+\infty$, so
\[
\iCLM\bigl(E^s_\lambda(+\infty),E^u_\lambda(-\infty);\lambda\in[0,\widehat{\lambda}]\bigr)=1
\]
and $\iMor(u)=\igeo(u)+1$.
\end{itemize}
\end{ex}


{\scriptsize{
\subsection*{Notation}

For the sake of the reader, let us introduce some common notations that we shall use henceforth throughout the paper.
\begin{itemize}
\item $\overline \R\=\R\cup\{\-\infty, +\infty\}$, $\R^+\=[0,+\infty)$,  $\R^-\=(-\infty,0]$. The pair $(\R^n, \langle \cdot, \cdot \rangle)$ denotes the $n$-dimensional Euclidean space  
\item  $\dot{\#}$ stands for denoting  the  derivative of $\#$ with respect to the time variable $t$
\item $\Id_X$ or just $\Id$ will denote the identity operator on a space $X$ and we set for simplicity $I_k := \Id_{\R^k}$ for $k \in\N$
\item $T\R^n\cong \R^n \times \R^n$ denotes the {\em tangent of $\R^n$\/} and $T^*\R^n\cong \R^n \times \R^n$ the {\em cotangent of $\R^n$\/}. $\omega$ stands for the {\em standard symplectic form\/} and the pair $(T^* \R^n, \omega)$ denotes the standard symplectic space. $J\= \begin{pmatrix}
	0 & -\Id\\ \Id & 0
\end{pmatrix}$ denotes the {\em standard symplectic matrix\/} and $\omega(u,v)=\langle Ju, v\rangle$.
\item $\Lagr(n)$ denotes the {\em Lagrangian Grassmannian manifold\/}. $L_D\=\R^n\times \{0\}$ and $L_N=\{0\}\times \R^n$ and we refer to as {\em Dirichlet\/}   and {\em Neumann Lagrangian subspace\/} 
\item $\Mat(n,\R)$ the set of all $n\times n$ matrices; $\Sym(n)$  the set of all $n \times n$ symmetric matrices, $\Sym^+(n)$  the set of all $n \times n$ positive definite and symmetric matrices. $V^+$ and $V^-$ denotes the positive and negative spectral spaces, respectively. $E^s, E^u$ the stable and unstable space respectively.
\item Given the linear subspaces $L_0,L_1$ we write  $L_0 \pitchfork L_1$ meaning that  $L_0 \cap L_1=\{0\}$. 
\item   $\big(\mathcal{H}, (\cdot, \cdot )\big)$ denotes  a real separable  Hilbert space.  $\Lin(\mathcal{H})$ denotes the Banach space of all bounded and linear operators.  $\mathcal{C}^{sa}(\mathcal{H})$ be the set of all (closed) densely defined and selfadjoint operators. We denote by  $\cfsa(\mathcal{H})$ the space of all closed selfadjoint and Fredholm operators equipped with the {\em gap topology\/}. $\sigma(\#)$ denotes the spectrum of the linear operator $\#$.   $\Sf{\#}$ denotes the {\em spectral flow\/} of the path of selfadjoint Fredholm operators $\#$. $\iMor(\#)$ denotes the Morse index. $V^{\pm}(\#)$ denotes the positive and negative spectral space of $\#$. $\im(\#)$ stands for denoting the  image of the operator $\#$. $\iMor (u)$ (resp.  $\iMor(u, L_0,\pm)$) are the Morse indices of the heteroclinic $u$ (resp. future or past  halfclinic orbit $u$).
\item  $\iCLM$-denotes the {\em Maslov index\/} of a pair of Lagrangian paths. $\iota(\#_1,\#_2,\#_3)$ denotes the {\em triple index\/}.  $\igeo(u)$ (resp. $\igeo^\pm$) are the geometrical indices for heteroclinic (resp. future or past halfclinic orbit)
\item $\mathcal A^\pm:=\mathcal A\big\vert_{W^{2,2}(\R^\pm, \R^n)}$ and $\mathcal F^\pm:=\mathcal A\big\vert_{W^{1,2}(\R^\pm, \R^n)}$

\item $\mathcal A_m$ and $\mathcal A_m^\pm$   the {\bf minimal operators} associated to $\mathcal A$ and $\mathcal A^\pm$, respectively

 \item $\mathcal F_m$ and $\mathcal F_m^\pm$   the {\bf minimal operators} associated to $\mathcal F$ and $\mathcal F^\pm$, respectively
\end{itemize}
}}


\section{Fredholmness, hyperbolicity and spectral flows}

In this section we recall the relation between Fredholm properties of the Sturm–Liouville and Hamiltonian realizations and the hyperbolicity of the limiting matrices, and we collect the spectral flow formulas needed in the sequel. Technical details for Sturm–Liouville operators are deferred to Appendix~\ref{sec:fredholm-sturm}.

\subsection{Fredholmness and hyperbolicity}

We begin with the half–line case.

\begin{lem}\label{thm:s.l.+ fredholm iff hyperbolic}
The operator $\mathcal A^+_{L_0}$ (resp.\ $\mathcal A^-_{L_0}$) is Fredholm if and only if $JB(+\infty)$ (resp.\ $JB(-\infty)$) is hyperbolic.
\end{lem}

\begin{proof}
By \cite[Chapter~IV, Theorem~5.35]{Kat80}, Lemma~\ref{lem:relative compact perturbation} and Corollary~\ref{lem:min s.l.o. fredhom iff hyperbolic}, $\mathcal F^+_m$ is Fredholm if and only if $JB(+\infty)$ is hyperbolic. The claim follows from Lemma~\ref{thm:fredholm-equivalent} and Lemma~\ref{thm:equivalence-SL}.
\end{proof}

Consider now the operator
\[
\widetilde{\mathcal A}_{L_0}:=\mathcal A^-_{L_0}\oplus\mathcal A^+_{L_0}
\]
with domain $\dom\widetilde{\mathcal A}_{L_0}=\dom\mathcal A^-_{L_0}\oplus \dom\mathcal A^+_{L_0}\subset L^{2}(\R^-,\R^n)\oplus L^{2}(\R^+,\R^n)$, and set
\[
\widetilde E:=\Bigl\{(u,v)\in W^{2,2}(\R^-,\R^n)\oplus W^{2,2}(\R^+,\R^n)\,\Big|\,
\begin{pmatrix}\dot u(0)\\u(0)\end{pmatrix}
=
\begin{pmatrix}\dot v(0)\\v(0)\end{pmatrix}
\Bigr\}.
\]
Then $\mathcal A$ is the restriction of $\widetilde{\mathcal A}_{L_0}$ to $\widetilde E$.

\begin{lem}\label{thm:s.l. fredholm iff hyperbolic}
The operator $\mathcal A$ is Fredholm if and only if $JB(\pm\infty)$ are both hyperbolic.
\end{lem}

\begin{proof}
Assume $\mathcal A$ is Fredholm. Since $\mathcal A=\widetilde{\mathcal A}_{L_0}|_{\widetilde E}$, we have
\[
\codim\im\mathcal A^+_{L_0} + \codim\im\mathcal A^-_{L_0}
= \codim\im\widetilde{\mathcal A}_{L_0}
\le \codim\im\mathcal A <\infty.
\]
Thus $\codim\im\mathcal A^\pm_{L_0}<\infty$. Lemma~\ref{lem:finite-codim-rge-closed} implies that $\im\mathcal A^\pm_{L_0}$ are closed, hence $\mathcal A^\pm_{L_0}$ are Fredholm. By Lemma~\ref{thm:s.l.+ fredholm iff hyperbolic}, $JB(\pm\infty)$ are hyperbolic.

Conversely, if $JB(\pm\infty)$ are hyperbolic, then \cite{RS95} yields that $\mathcal A$ is Fredholm.
\end{proof}

We now pass to the associated first order Hamiltonian operators. Set
\[
\mathscr F:=-J\frac{d}{dt}-B(t),\qquad t\in I,
\]
and define
\[
W:=W^{1,2}(\R,\R^{2n}),\qquad
W^\pm_{L_0}:=\{z\in W^{1,2}(\R^\pm,\R^{2n})\mid z(0)\in L_0\}.
\]
We then put $\mathcal F:=\mathscr F|_W$ and $\mathcal F^\pm_{L_0}:=\mathscr F|_{W^\pm_{L_0}}$.

\begin{prop}\label{lem:fredholm}
With the above notation,
\begin{align*}
\mathcal A \text{ Fredholm } &\iff \mathcal F \text{ Fredholm } \iff JB(\pm\infty)\text{ hyperbolic},\\
\mathcal A^{\pm}_{L_0} \text{ Fredholm } &\iff \mathcal F^\pm_{L_0} \text{ Fredholm } \iff JB(\pm\infty)\text{ hyperbolic}.
\end{align*}
\end{prop}

\begin{proof}
The equivalence between hyperbolicity and Fredholmness for $\mathcal A$ and $\mathcal A_{L_0}^\pm$ follows from Lemmas~\ref{thm:s.l.+ fredholm iff hyperbolic} and \ref{thm:s.l. fredholm iff hyperbolic}. The equivalence between Fredholmness of $\mathcal A$ and $\mathcal F$ is proved in \cite{RS95}; the corresponding statement for $\mathcal A^\pm_{L_0}$ and $\mathcal F^\pm_{L_0}$ follows from \cite{RS05a,RS05b}.
\end{proof}

\subsection{Spectral flows}

We now construct suitable deformations of the Hamiltonian boundary value problem. Let $[0,1]\ni\lambda\mapsto R_\lambda(t)\in\Sym(n)$ be a continuous path and define
\begin{equation}\label{eq:Hamiltonian-lambda}
\dot z=JB_\lambda(t)z,\qquad
B_\lambda(t):=
\begin{pmatrix}
P^{-1}(t) & -P^{-1}(t)Q(t)\\
-Q(t)^\top P^{-1}(t)
& Q(t)^\top P^{-1}(t)Q(t)-R_\lambda(t)
\end{pmatrix}.
\end{equation}
Set $H_\lambda(t):=JB_\lambda(t)$.

\begin{note}
We denote by $\mathcal A_\lambda$, $\mathcal A_{L_0,\lambda}^\pm$, $\mathcal F_\lambda$, $\mathcal F_{L_0,\lambda}^\pm$ the operators obtained from $\mathcal A$, $\mathcal A_{L_0}^\pm$, $\mathcal F$, $\mathcal F_{L_0}^\pm$ by replacing $R$ with $R_\lambda$.
\end{note}

We assume:

\begin{itemize}
\item[\bf(H3)] There exist continuous paths of hyperbolic Hamiltonian matrices $\lambda\mapsto H_\lambda(\pm\infty)$ such that
\[
H_\lambda(+\infty)=\lim_{t\to+\infty}JB_\lambda(t),\qquad
H_\lambda(-\infty)=\lim_{t\to-\infty}JB_\lambda(t)
\]
uniformly in $\lambda$.
\end{itemize}

Under (H3) and Proposition~\ref{lem:fredholm}, all operators $\mathcal A_\lambda$, $\mathcal A_{L_0,\lambda}^\pm$, $\mathcal F_\lambda$, $\mathcal F_{L_0,\lambda}^\pm$ are selfadjoint Fredholm with dense domain, so their spectral flows are well-defined.

\begin{prop}\label{thm:spectral-flow-equal}
If {\rm(H3)} holds, then
\begin{multline}\label{eq:sf equal1}
\Sf{\mathcal A_\lambda;\lambda\in[0,1]}
=
\Sf{\mathcal F_\lambda;\lambda\in[0,1]},\\
\Sf{\mathcal A^+_{L_0,\lambda};\lambda\in[0,1]}
=
\Sf{\mathcal F_{L_0,\lambda}^{+};\lambda\in[0,1]},\\
\Sf{\mathcal A^-_{L_0,\lambda};\lambda\in[0,1]}
=
\Sf{\mathcal F_{L_0,\lambda}^{-};\lambda\in[0,1]}.
\end{multline}
\end{prop}

\begin{proof}
We prove only the first equality in \eqref{eq:sf equal1}; the other two follow analogously.

\smallskip\noindent
\emph{Step 1. Continuity.}
By (H3) and standard perturbation theory for selfadjoint operators (see \cite[Chapter~4, Section~6, Theorem~2.24]{Kat80}), $\lambda\mapsto\mathcal A_\lambda$ and $\lambda\mapsto\mathcal F_\lambda$ are continuous paths in the gap topology, hence have well-defined spectral flows.

\smallskip\noindent
\emph{Step 2. A comparison homotopy.}
Consider the two–parameter family
\[
h(\lambda,s):=-J\frac{d}{dt}-B_{\lambda,s}(t),
\qquad (\lambda,s)\in[0,1]\times[0,\delta],
\]
with
\[
B_{\lambda,s}(t):=
\begin{pmatrix}
P^{-1}(t) & -P^{-1}(t)Q(t)\\
-Q(t)^\top P^{-1}(t)
& Q(t)^\top P^{-1}(t)Q(t)-R_\lambda(t)-s\,\Id
\end{pmatrix}.
\]
For each fixed $\lambda$, the path $s\mapsto h(\lambda,s)$ is a positive path in the sense of Definition~\ref{def:positive curve}.

\smallskip\noindent
\emph{Step~3. Local comparison in  $s$.}
By virtue of the bijection
\[
\ker h(\lambda,s) \ni w
\;\longmapsto\;
\begin{pmatrix}
P \dot{w}+Qw\\[2pt]
w
\end{pmatrix}
\in
\ker(\mathcal A_{\lambda}+s\Id),
\]
in order to compute the local contribution to the spectral flow it suffices to show that,
for each fixed $\lambda$, the paths
\[
s \longmapsto h(\lambda,s)
\quad \text{and} \quad
s \longmapsto \mathcal A_{\lambda}+s\Id
\]
are positive.

In fact, for every $w \in \ker h(\lambda,s)$ we have
\[
\left\langle \frac{d}{ds} h(\lambda,s) w,\, w \right\rangle
=
\langle w, w \rangle,
\]
and, correspondingly,
\[
\left\langle
\frac{d}{ds}\bigl(\mathcal A_{\lambda}+s\Id\bigr)
\begin{pmatrix}
P \dot{w}+Qw\\[2pt]
w
\end{pmatrix},
\begin{pmatrix}
P \dot{w}+Qw\\[2pt]
w
\end{pmatrix}
\right\rangle
=
\langle w, w \rangle.
\]
Hence both paths are positive at every crossing.

As a consequence, the local spectral flows coincide and we obtain
\[
\Sf{h(\lambda,s);\, s\in[0,\delta]}
=
\sum_{0<s\le\delta}\dim\ker h(\lambda,s)
=
\sum_{0<s\le\delta}\dim\ker(\mathcal A_{\lambda}+s\Id)
=
\Sf{\mathcal A_{\lambda}+s\Id;\, s\in[0,\delta]}.
\]

\smallskip\noindent
\emph{Step 4. Local comparison in $\lambda$.}
Let $\lambda_0$ be a crossing of $\lambda\mapsto\mathcal A_\lambda$. Since $0$ is isolated in the spectrum, there exists $\delta>0$ such that
\[
\ker(\mathcal A_{\lambda_0}+\delta\Id) = \{0\},\qquad
\ker h(\lambda_0,\delta) = \{0\}.
\]
By openness of the set of invertible selfadjoint Fredholm operators, there is $\delta_1>0$ such that
\[
\ker(\mathcal A_\lambda+\delta\Id)=\{0\},\qquad
\ker h(\lambda,\delta)=\{0\}
\]
for all $\lambda\in[\lambda_0-\delta_1,\lambda_0+\delta_1]$. Thus
\[
\Sf{\mathcal A_\lambda+\delta\Id;\ \lambda\in[\lambda_0-\delta_1,\lambda_0+\delta_1]}
=
\Sf{h(\lambda,\delta);\ \lambda\in[\lambda_0-\delta_1,\lambda_0+\delta_1]}
=
0.
\]

Homotopy invariance gives
\begin{equation}\label{eq:spectral flow A}
\Sf{\mathcal A_\lambda;\ \lambda\in[\lambda_0-\delta_1,\lambda_0+\delta_1]}
=
\Sf{\mathcal A_{\lambda_0-\delta_1}+s\Id;\ s\in[0,\delta]}
-
\Sf{\mathcal A_{\lambda_0+\delta_1}+s\Id;\ s\in[0,\delta]},
\end{equation}
and
\begin{equation}\label{eq:spectral flow F}
\Sf{h(\lambda,0);\ \lambda\in[\lambda_0-\delta_1,\lambda_0+\delta_1]}
=
\Sf{h(\lambda_0-\delta_1,s);\ s\in[0,\delta]}
-
\Sf{h(\lambda_0+\delta_1,s);\ s\in[0,\delta]}.
\end{equation}
Since each $s$–path is positive,
\begin{align}\label{eq:sf A= sf F}
\Sf{\mathcal A_{\lambda_0\pm\delta_1}+s\Id;\ s\in[0,\delta]}
&=
\sum_{0<s\le\delta}\dim\ker(\mathcal A_{\lambda_0\pm\delta_1}+s\Id)\\
&=
\sum_{0<s\le\delta}\dim\ker h(\lambda_0\pm\delta_1,s)
=
\Sf{h(\lambda_0\pm\delta_1,s);\ s\in[0,\delta]}.\nonumber
\end{align}
Combining \eqref{eq:spectral flow A}, \eqref{eq:spectral flow F} and \eqref{eq:sf A= sf F} yields the local identity
\begin{equation}\label{eq:sf=local}
\Sf{\mathcal A_\lambda;\ \lambda\in[\lambda_0-\delta_1,\lambda_0+\delta_1]}
=
\Sf{h(\lambda,0);\ \lambda\in[\lambda_0-\delta_1,\lambda_0+\delta_1]}
=
\Sf{\mathcal F_\lambda;\ \lambda\in[\lambda_0-\delta_1,\lambda_0+\delta_1]}.
\end{equation}

\smallskip\noindent
\emph{Step 6. Global conclusion.}
By additivity of the spectral flow under concatenation, applying \eqref{eq:sf=local} along a partition of $[0,1]$ yields
\[
\Sf{\mathcal A_\lambda;\lambda\in[0,1]}
=
\Sf{\mathcal F_\lambda;\lambda\in[0,1]}.
\]
\end{proof}

Let $\gamma_{(\tau,\lambda)}$ be the fundamental solution of \eqref{eq:Hamiltonian-lambda} and denote by $E_\lambda^{s/u}(\tau)$ the associated stable/unstable spaces, and by $E_\lambda^{s}(+\infty)$ and $E_\lambda^{u}(-\infty)$ the asymptotic stable and unstable spaces. Under (H3),
\[
\lim_{\tau\to+\infty}E_\lambda^{s}(\tau)=E_\lambda^{s}(+\infty),
\qquad
\lim_{\tau\to-\infty}E_\lambda^{u}(\tau)=E_\lambda^{u}(-\infty)
\]
in the gap topology (see \cite{AM03}).

\begin{prop}\label{cor: S-L-operator=index}
Under {\rm(H3)}, the following identities hold:
\begin{align}\label{eq:index-formula-1-of-2}
\Sf{\mathcal A_\lambda;\lambda\in[0,1]}
=&\ \iCLM\big(E_1^{s}(\tau),E_1^{u}(-\tau);\tau\in\R^+\big)
-\iCLM\big(E_0^{s}(\tau),E_0^{u}(-\tau);\tau\in\R^+\big)\nonumber\\
&\ -\iCLM\big(E_\lambda^{s}(+\infty),E_\lambda^{u}(-\infty);\lambda\in[0,1]\big),
\end{align}
and
\begin{align}
\Sf{\mathcal A_\lambda^{+};\lambda\in[0,1]}
=&\ \iCLM\big(E_1^{s}(\tau),L_0;\tau\in\R^+\big)
-\iCLM\big(E_0^{s}(\tau),L_0;\tau\in\R^+\big)\nonumber\\
&\ -\iCLM\big(E_\lambda^{s}(+\infty),L_0;\lambda\in[0,1]\big),\\
\Sf{\mathcal A_\lambda^{-};\lambda\in[0,1]}
=&\ \iCLM\big(L_0,E_1^{u}(-\tau);\tau\in\R^+\big)
-\iCLM\big(L_0,E_0^{u}(-\tau);\tau\in\R^+\big)\nonumber\\
&\ -\iCLM\big(L_0,E_\lambda^{u}(-\infty);\lambda\in[0,1]\big).
\end{align}
\end{prop}

\begin{proof}
This follows directly from \cite[Theorem~1]{HP17} together with Proposition~\ref{thm:spectral-flow-equal}.
\end{proof}

\begin{rem}\label{rmk:spectral-flow-morse-indices-difference}
Assume (L1), (H1) and (H2) and consider the path $\lambda\mapsto\mathcal A_\lambda:=\mathcal A+\lambda\Id$. As shown in Section~\ref{sec:fredholm-sturm}, $\lambda\mapsto\mathcal A_\lambda$ is a positive path of selfadjoint Fredholm operators and there exists $\widehat\lambda>0$ such that $\ker\mathcal A_\lambda=\{0\}$ for all $\lambda\ge\widehat\lambda$. For such positive paths it is well known that the spectral flow equals the difference of Morse indices between the endpoints. Hence
\begin{equation}\label{eq:Morse sf flow}
\iMor(\mathcal A)=\Sf{\mathcal A_\lambda;\lambda\in[0,\widehat\lambda]},
\qquad
\iMor(\mathcal A^\pm)=\Sf{\mathcal A^\pm_\lambda;\lambda\in[0,\widehat\lambda]}.
\end{equation}
\end{rem}


\section{Transversality between invariant subspaces}

In this section we give sufficient conditions on the coefficients of the Sturm–Liouville operators to ensure non-degeneracy of the associated operators. This guarantees that the various indices are well defined.

\subsection{Transversality for heteroclinics}

\begin{lem}\label{lem:matrix positive condition}
Assume {\rm(L1)}–{\rm(L2)}. Then
\[
K_\lambda(t)=
\begin{pmatrix}
P(t) & Q(t)\\
Q(t)^T & R(t)+\lambda\Id
\end{pmatrix}
\]
is positive definite for all $(t,\lambda)\in\R\times\bigl[\frac{8C_2^2}{C_1}+C_3,\infty\bigr)$, where $C_1,C_2,C_3$ are the constants in {\rm(L2)}.
\end{lem}

\begin{proof}
For $(u,v)\in\R^n\times\R^n$ we have, using {\rm(L2)} and Cauchy–Schwarz,
\[
\langle K_\lambda(t)(u,v),(u,v)\rangle
\ge C_1|u|^2 - 2 C_2 |u||v| + (\lambda - C_3)|v|^2 .
\]
By Young’s inequality, for any $\varepsilon>0$,
\[
2C_2|u||v|\le 2C_2\bigl(\varepsilon|u|^2 + \varepsilon^{-1}|v|^2\bigr),
\]
hence
\[
\langle K_\lambda(t)(u,v),(u,v)\rangle
\ge (C_1-2\varepsilon C_2)|u|^2 + (\lambda - C_3 - 2C_2/\varepsilon)|v|^2 .
\]
Taking $\varepsilon=C_1/(4C_2)$ yields $C_1-2\varepsilon C_2=C_1/2>0$ and $2C_2/\varepsilon=8C_2^2/C_1$; the form is positive for all $(u,v)\neq0$ when $\lambda\ge \frac{8C_2^2}{C_1}+C_3$.
\end{proof}

\begin{lem}\label{lem:matrx define*}
Assume {\rm(L1)}–{\rm(L2)} and let $D>0$ with $D\le C_0$. Then
\[
M(t,\lambda)=
\begin{pmatrix}
P(t) & Q(t)+D\Id\\
Q(t)^T+D\Id & R(t)+\lambda\Id
\end{pmatrix}
\]
is positive definite for all $(t,\lambda)\in\R\times\bigl[\frac{8(C_2+C_0)^2}{C_1}+C_3,\infty\bigr)$.
\end{lem}

\begin{proof}
As in Lemma~\ref{lem:matrix positive condition}, for $z=(x,y)$,
\[
\langle M(t,\lambda)z,z\rangle
\ge C_1|x|^2 -2(C_2+D)|x||y| + (\lambda-C_3)|y|^2.
\]
Using $2ab\le a^2+b^2$ with $a=\sqrt{C_1}|x|$ and $b=(C_2+D)|y|/\sqrt{C_1}$ gives
\[
2(C_2+D)|x||y|
\le C_1|x|^2 + \frac{(C_2+D)^2}{C_1}|y|^2,
\]
so
\[
\langle M(t,\lambda)z,z\rangle
\ge\Bigl(\lambda - C_3 - \tfrac{(C_2+D)^2}{C_1}\Bigr)|y|^2.
\]
Since $(C_2+D)^2\le (C_2+C_0)^2$, the right-hand side is positive if 
$\lambda\ge \frac{8(C_2+C_0)^2}{C_1}+C_3$; and for $y=0$ we have
$\langle M(t,\lambda)z,z\rangle\ge C_1|x|^2>0$.
\end{proof}

For $s>0$ consider the operators
\[
\mathcal A_{s,M}^{\pm,\lambda}u
=
-\frac{d}{dt}\bigl(P(t\!\pm\! s)u'(t)+Q(t\!\pm\! s)u(t)\bigr)
+Q(t\!\pm\! s)^T u'(t)
+R_\lambda(t\!\pm\! s)u(t),
\]
on $W^{2,2}(\R^\pm,\R^n)\subset L^2(\R^\pm,\R^n)$, where $R_\lambda(t)=R(t)+\lambda\Id$.

\begin{lem}\label{lem:system_future_past_1}
Assume {\rm(L1)}–{\rm(L2)}.  
If $\lambda\ge\frac{8C_2^2}{C_1}+C_3$, then the system
\[
\begin{cases}
\mathcal A_{s,M}^{+,\lambda}x_1(t)=0,\quad t>0,\\
\mathcal A_{s,M}^{-,\lambda}x_2(t)=0,\quad t<0,\\
x_1(0)=x_2(0),\\
P(s)\dot x_1(0)+Q(s)x_1(0)=P(-s)\dot x_2(0)+Q(-s)x_2(0)
\end{cases}
\]
admits only the trivial solution.
\end{lem}

\begin{proof}
Assume $(x_1,x_2)\not\equiv 0$ solves the system.  
Integrating by parts,
\[
\langle\mathcal A_{s,M}^{+,\lambda}x_1,x_1\rangle_{L^2}=I_1
+\langle P(s)\dot x_1(0)+Q(s)x_1(0),x_1(0)\rangle,
\]
\[
\langle\mathcal A_{s,M}^{-,\lambda}x_2,x_2\rangle_{L^2}
=I_2
-\langle P(-s)\dot x_2(0)+Q(-s)x_2(0),x_2(0)\rangle,
\]
where
\[
I_1=\int_0^\infty\!\Big(
\langle P(t\!+\!s)\dot x_1,\dot x_1\rangle
+\langle Q(t\!+\!s)x_1,\dot x_1\rangle
+\langle Q(t\!+\!s)^T\dot x_1,x_1\rangle
+\langle R_\lambda(t\!+\!s)x_1,x_1\rangle
\Big)\,dt,
\]
and $I_2$ is defined similarly on $(-\infty,0]$.  
Using the boundary condition at $t=0$ the boundary terms cancel, so $I_1+I_2=0$.

Set
\[
K_\lambda(t\pm s)=
\begin{pmatrix}
P(t\pm s) & Q(t\pm s)\\
Q(t\pm s)^T & R(t\pm s)+\lambda\Id
\end{pmatrix}.
\]
Then
\[
I_1=\int_0^\infty 
\begin{pmatrix}\dot x_1\\ x_1\end{pmatrix}^T
K_\lambda(t+s)
\begin{pmatrix}\dot x_1\\ x_1\end{pmatrix} dt,
\quad
I_2=\int_{-\infty}^0 
\begin{pmatrix}\dot x_2\\ x_2\end{pmatrix}^T
K_\lambda(t-s)
\begin{pmatrix}\dot x_2\\ x_2\end{pmatrix} dt.
\]
By Lemma~\ref{lem:matrix positive condition}, $K_\lambda$ is positive definite for $\lambda\ge\frac{8C_2^2}{C_1}+C_3$, hence $I_1,I_2\ge0$ and $I_1+I_2=0$ implies $x_1\equiv x_2\equiv 0$, a contradiction.
\end{proof}

\begin{cor}\label{thm:operator-non-degenerate}
Assume {\rm(L2)}. Then $\mathcal A_\lambda$ is non-degenerate for every 
\(
\lambda\ge \frac{8C_2^2}{C_1}+C_3.
\)
\end{cor}

\begin{proof}
A function $u\in W^{2,2}(\R,\R^n)$ lies in $\ker\mathcal A_\lambda$ iff its restrictions
\[
x_1(t):=u(t)\ (t>0),\qquad x_2(t):=u(t)\ (t<0)
\]
solve the system in Lemma~\ref{lem:system_future_past_1} with $s=0$ and the matching conditions 
$x_1(0)=x_2(0)$, $P(0)\dot x_1(0)+Q(0)x_1(0)=P(0)\dot x_2(0)+Q(0)x_2(0)$.  
Lemma~\ref{lem:system_future_past_1} then implies $u\equiv 0$.
\end{proof}

Now consider the first order operators $\mathcal F_{s,M}^{\pm,\lambda}$ associated to $\mathcal A_{s,M}^{\pm,\lambda}$.  
If $z=(p,x)$ solves $\mathcal F_{s,M}^{\pm,\lambda}z=0$, then $x\in W^{2,2}(\R^\pm,\R^n)$ solves $\mathcal A_{s,M}^{\pm,\lambda}x=0$ and
\[
p(t)=P(t\pm s)\dot x(t)+Q(t\pm s)x(t).
\]

\begin{lem}\label{lem:system_future_past}
Assume {\rm(L2)} and $\lambda \ge \frac{8C_2^2}{C_1}+C_3$. 
Then the initial value problem
\[
\begin{cases}
\mathcal F_{s,M}^{+,\lambda}z_1(t)=0,\\
\mathcal F_{s,M}^{-,\lambda}z_2(t)=0,\\
z_1(0)=z_2(0)
\end{cases}
\]
admits only the trivial solution.
\end{lem}

\begin{proof}
If $(z_1,z_2)\not\equiv (0,0)$ solves the system, write
\[
z_1(t)=(p_1(t),x_1(t)),\quad t>0;\qquad
z_2(t)=(p_2(t),x_2(t)),\quad t<0.
\]
Then $x_i$ solve $\mathcal A_{s,M}^{\pm,\lambda}x_i=0$ and
\[
p_1(t)=P(t+s)\dot x_1(t)+Q(t+s)x_1(t),\quad
p_2(t)=P(t-s)\dot x_2(t)+Q(t-s)x_2(t).
\]
The condition $z_1(0)=z_2(0)$ yields precisely the boundary conditions in Lemma~\ref{lem:system_future_past_1}. Hence $(x_1,x_2)$ is a nontrivial solution of that problem, contradicting Lemma~\ref{lem:system_future_past_1}.
\end{proof}

\begin{prop}\label{lem:index=0 at infinity:real line}
Under {\rm(L2)},
\[
E_\lambda^u(-\tau)\cap E_\lambda^s(\tau)=\{0\}
\quad\text{for all}\quad
(\tau,\lambda)\in\R^+\times\Bigl[\tfrac{8C_2^2}{C_1}+C_3,\infty\Bigr).
\]
\end{prop}

\begin{proof}
Fix $\tau>0$ and set $s=\tau$.  
The stable (resp.\ unstable) space at $t=0$ of $\mathcal F_{s,M}^{+,\lambda}$ (resp.\ $\mathcal F_{s,M}^{-,\lambda}$) is $E_\lambda^s(\tau)$ (resp.\ $E_\lambda^u(-\tau)$).  

Let $\mathcal S$ be the space of pairs of solutions
\[
(z_1,z_2),\qquad
\mathcal F_{s,M}^{+,\lambda}z_1=0\ \text{on }(0,\infty),\quad
\mathcal F_{s,M}^{-,\lambda}z_2=0\ \text{on }(-\infty,0),
\]
such that $z_1(0)=z_2(0)$.  
Define
\[
\Phi:\mathcal S\to E_\lambda^u(-\tau)\cap E_\lambda^s(\tau),\qquad
\Phi(z_1,z_2)=z_1(0)=z_2(0).
\]
By uniqueness of solutions in stable/unstable manifolds, $\Phi$ is an isomorphism.  
By Lemma~\ref{lem:system_future_past}, for 
$\lambda\ge\frac{8C_2^2}{C_1}+C_3$ the only such pair is $(0,0)$, so 
$E_\lambda^u(-\tau)\cap E_\lambda^s(\tau)=\{0\}$.
\end{proof}

\subsection{Transversality for the halfclinic case}

We now treat the half-line case with general Lagrangian boundary conditions.

Let $(y,x)^T\in L_0$, and denote by $L_N:=\{0\}\times\R^n$ the Neumann Lagrangian.  
Set
\[
V(L_0):=(L_0+L_D)\cap L_N.
\]
Elements of $V(L_0)$ are of the form $(0,x)$.  
Using the orthogonal decomposition $\R^n=V(L_0)\oplus V(L_0)^\perp$, write $y=y_1+y_2$ with $y_1\in V(L_0)$, $y_2\in V(L_0)^\perp$.  
In a basis of $V(L_0)$ there is a symmetric matrix $A$ such that $y_1=Ax$, hence
\[
\langle y,x\rangle=\langle Ax,x\rangle
\]
for all $(y,x)\in L_0$.  
Thus there exists $C_0>0$ such that
\begin{equation}\label{eq:estimation for L_0}
|\langle Ax,x\rangle|\le C_0|x|^2\qquad\text{for all }(0,x)\in V(L_0),
\end{equation}
and in particular
\begin{equation}\label{eq:estimate_(x,y)}
|\langle y,x\rangle|\le C_0|x|^2\qquad\text{for all }(y,x)\in L_0.
\end{equation}

\begin{lem}\label{lem:operator non-degenerate+}
If {\rm(L2)} holds, then $\mathcal A_{L_0,\lambda}^{\pm}$ is non-degenerate for every 
\[
\lambda\ge \frac{8(C_2+C_0)^2}{C_1}+C_3,
\]
where $C_0$ is given by \eqref{eq:estimation for L_0}.
\end{lem}

\begin{proof}
We treat $\mathcal A_{L_0,\lambda}^+$; the minus case is similar.  

Let $x\in\ker\mathcal A_{L_0,\lambda}^+$. Then
\[
\begin{cases}
\mathcal A_{M,\lambda}^+ x=0,\\[2pt]
\bigl(P(0)\dot x(0)+Q(0)x(0),\, x(0)\bigr)^T \in L_0,
\end{cases}
\]
where
\[
\mathcal A_{M,\lambda}^+
=
-\frac{d}{dt}\bigl(P(t)x'(t)+Q(t)x(t)\bigr)
+Q(t)^T x'(t)
+(R(t)+\lambda\Id)x(t).
\]

Integration by parts gives
\[
\begin{aligned}
\langle \mathcal A_{M,\lambda}^+ x, x\rangle_{L^2}
&=
\langle P\dot x,\dot x\rangle_{L^2}
+\langle Qx,\dot x\rangle_{L^2}
+\langle Q^T\dot x,x\rangle_{L^2}
+\langle (R+\lambda\Id)x,x\rangle_{L^2} \\
&\quad
+\langle P(0)\dot x(0)+Q(0)x(0),x(0)\rangle \\
&\ge
\langle P\dot x,\dot x\rangle_{L^2}
+\langle Qx,\dot x\rangle_{L^2}
+\langle Q^T\dot x,x\rangle_{L^2}
+\langle Rx,x\rangle_{L^2}
-C_0|x(0)|^2,
\end{aligned}
\]
using \eqref{eq:estimate_(x,y)}.

Moreover
\[
|x(0)|^2=-\int_0^\infty\frac{d}{dt}|x(t)|^2\,dt
= -2\langle \dot x,x\rangle_{L^2},
\]
so
\[
\begin{aligned}
0
&=\langle \mathcal A_{M,\lambda}^+ x, x\rangle_{L^2} \\
&\ge 
\langle P\dot x,\dot x\rangle_{L^2}
+\langle (Q+C_0\Id)x,\dot x\rangle_{L^2}
+\langle (Q^T+C_0\Id)\dot x,x\rangle_{L^2}
+\langle Rx,x\rangle_{L^2}.
\end{aligned}
\]
Equivalently,
\[
0 \ge \int_0^\infty 
\begin{pmatrix}\dot x(t) \\ x(t)\end{pmatrix}^T
\begin{pmatrix}
P(t) & Q(t)+C_0\Id \\
Q(t)^T+C_0\Id & R(t)+\lambda\Id
\end{pmatrix}
\begin{pmatrix}\dot x(t) \\ x(t)\end{pmatrix}
dt.
\]
By Lemma~\ref{lem:matrx define*}, the matrix in the integral is positive definite for all 
$\lambda\ge \frac{8(C_2+C_0)^2}{C_1}+C_3$, hence the integrand vanishes only when $x\equiv0$.
\end{proof}

Arguing as in Proposition~\ref{lem:index=0 at infinity:real line} and using 
Lemma~\ref{lem:operator non-degenerate+}, we obtain the corresponding statement for the
half-line boundary condition.

\begin{lem}\label{lem:index=0 at infinity:half line}
Assume {\rm(L1)}–{\rm(L2)}. For every 
\[
(\tau,\lambda)\in\R^+\times\Bigl[\tfrac{8(C_2+C_0)^2}{C_1}+C_3,\infty\Bigr)
\]
we have
\[
E_\lambda^u(-\tau)\cap L_0=\{0\},
\qquad 
E_\lambda^s(\tau)\cap L_0=\{0\}.
\]
\end{lem}

\begin{proof}
We prove $E_\lambda^u(-\tau)\cap L_0=\{0\}$; the other case is analogous.

Let $v\in E_\lambda^u(-\tau)\cap L_0$.  
There exists a unique solution $z(t)=(p(t),x(t))$ of $\mathcal F_{s,M}^{-,\lambda}z=0$ on $(-\infty,0]$
with $s=\tau$, $z(0)=v$, and $z(t)\to0$ as $t\to-\infty$.  
Then $x$ solves $\mathcal A_{s,M}^{-,\lambda}x=0$ and
\[
\bigl(p(0),x(0)\bigr)^T
=
\bigl(P(-\tau)\dot x(0)+Q(-\tau)x(0),x(0)\bigr)\in L_0.
\]
Thus $x$ solves the boundary problem defining $\mathcal A_{L_0,\lambda}^-$.  
By Lemma~\ref{lem:operator non-degenerate+}, $x\equiv0$, so $v=z(0)=0$.
\end{proof}

\subsection{CLM-index of the (un)stable paths at infinity}

We now compute
\[
\iCLM\bigl(E^s_\lambda(+\infty),E^u_\lambda(-\infty);\ \lambda\in[0,\widehat\lambda]\bigr),
\]
via the triple indices
\[
\iota\bigl(E^u_{\widehat\lambda}(-\infty),E^s_{\widehat\lambda}(+\infty);L_D\bigr),
\qquad
\iota\bigl(E^u_{0}(-\infty),E^s_{0}(+\infty);L_D\bigr).
\]

Since $E_\lambda^{s/u}(\pm\infty)$ are Lagrangian and transversal to $L_D$, they admit the graph
representations
\[
E_\lambda^u(-\infty)=
\left\{
\begin{pmatrix}
N_\lambda u\\
u
\end{pmatrix}
: u\in\R^n
\right\},
\qquad
E_\lambda^s(+\infty)=
\left\{
\begin{pmatrix}
M_\lambda u\\
u
\end{pmatrix}
: u\in\R^n
\right\},
\]
with symmetric $M_\lambda,N_\lambda$.

For $u\in\R^n$,
\[
\begin{pmatrix}
N_\lambda u\\ u
\end{pmatrix}
=
\begin{pmatrix}
M_\lambda u\\ u
\end{pmatrix}
+
\begin{pmatrix}
(N_\lambda-M_\lambda)u\\ 0
\end{pmatrix},
\]
and from the definition of the triple form and \eqref{eq:trip1},
\begin{align}\label{eq:expression of Q}
&Q\!\left(E_\lambda^u(-\infty),E_\lambda^s(+\infty),L_D\right)
\Biggl(
\begin{pmatrix}N_\lambda u\\ u\end{pmatrix},
\begin{pmatrix}(N_\lambda-M_\lambda)u\\ 0\end{pmatrix}
\Biggr) \notag\\
&\qquad
=\left\langle 
\begin{pmatrix}
0 & -\Id \\ \Id & 0
\end{pmatrix}
\begin{pmatrix}
M_\lambda u \\ u
\end{pmatrix},
\begin{pmatrix}
(N_\lambda-M_\lambda)u \\ 0
\end{pmatrix}
\right\rangle 
=
\langle (M_\lambda-N_\lambda)u,\,u\rangle.
\end{align}
Thus the sign of the triple index is determined by $M_\lambda-N_\lambda$.

Consider the asymptotic first order system
\[
\mathscr F_\lambda^{+\infty}
:=-J\frac{d}{dt}-B_\lambda(+\infty),
\]
and its second order operator $\mathcal A_\lambda^{+\infty}$ (maximal realization 
$\mathcal A_{\lambda,M}^{+\infty}$).  
For $x\in\ker \mathcal A_{\lambda,M}^{+\infty}$, the map
\[
x\longmapsto
\bigl(P_+ \dot x(0)+Q_+ x(0),\,x(0)\bigr)
\]
induces a bijection between $\ker\mathcal A_{\lambda,M}^{+\infty}$ and 
\[
E_\lambda^s(0)=V^-\bigl(JB_\lambda(+\infty)\bigr).
\]
A direct computation yields
\[
\begin{aligned}
0
&=\langle \mathcal A_{\lambda,M}^{+\infty} x,\,x\rangle_{L^2} \\
&=\langle P_+ \dot x,\dot x\rangle_{L^2}
  +\langle Q_+ x,\dot x\rangle_{L^2}
  +\langle Q_+^T \dot x,x\rangle_{L^2}
  +\langle (R_+ +\lambda\Id)x,x\rangle_{L^2}
  +\langle M_\lambda x(0),x(0)\rangle .
\end{aligned}
\]
For $v=x(0)$,
\begin{equation}\label{eq:M negative}
\langle M_\lambda v,v\rangle
=
-\int_0^\infty
\Biggl\langle 
K_\lambda
\begin{pmatrix}\dot x(t) \\ x(t)\end{pmatrix},
\begin{pmatrix}\dot x(t) \\ x(t)\end{pmatrix}
\Biggr\rangle dt,
\qquad
K_\lambda:=
\begin{pmatrix}
P_+ & Q_+ \\ Q_+^T & R_+ +\lambda\Id
\end{pmatrix}.
\end{equation}
An analogous formula holds for $N_\lambda$ at $-\infty$ with an integral over $(-\infty,0]$.

Whenever $K_\lambda$ is positive definite (by Lemma~\ref{lem:matrix positive condition}, this holds for all $\lambda\ge \frac{8C_2^2}{C_1}+C_3$), 
\[
\langle M_\lambda v,v\rangle<0,\qquad
\langle N_\lambda v,v\rangle>0
\quad (v\neq0).
\]

\begin{lem}\label{thm:M0N0}
If {\rm(H2)} holds, then $M_0$ is negative definite and $N_0$ is positive definite. 
If also {\rm(L1)}–{\rm(L2)} hold, then for all
\[
\lambda \ge \frac{8C_2^2}{C_1} + C_3
\]
the matrices $M_\lambda$ and $N_\lambda$ are respectively negative and positive definite.
\end{lem}

\begin{proof}
Under {\rm(H2)}, $K_0$ is positive definite at $\pm\infty$, so \eqref{eq:M negative} (and its backward analogue) give $M_0<0$ and $N_0>0$.  
If {\rm(L1)}–{\rm(L2)} hold, Lemma~\ref{lem:matrix positive condition} implies $K_\lambda>0$ for all 
$\lambda \ge \frac{8C_2^2}{C_1} + C_3$, hence $M_\lambda<0$ and $N_\lambda>0$.
\end{proof}

\begin{lem}\label{thm:epressionQ+}
Assume {\rm(L1)}–{\rm(L2)} and let $(u,v)\in L_0$. Then
\[
\langle M_\lambda v, v\rangle - \langle u, v\rangle \le 0 
\qquad\text{for all}\qquad 
\lambda \ge \frac{8(C_2+C_0)^2}{C_1} + C_3.
\]
\end{lem}

\begin{proof}
Let $x\in\ker \mathcal A_{\lambda,M}^{+\infty}$ with $x(0)=v$.  
From the previous computation one obtains
\[
\langle M_\lambda v, v\rangle - \langle u, v\rangle
\le
-\int_0^\infty 
\Biggl\langle 
\Bigl(K_\lambda + C_0 
\begin{pmatrix} 
0 & \Id \\ \Id & 0 
\end{pmatrix}\Bigr)
\binom{\dot x(t)}{x(t)},\,
\binom{\dot x(t)}{x(t)}
\Biggr\rangle dt.
\]
By Lemma~\ref{lem:matrx define*}, the matrix in parentheses is positive definite for 
$\lambda \ge \frac{8(C_2+C_0)^2}{C_1} + C_3$, so the integral is non-positive.
\end{proof}

\begin{rem}\label{rem:epression Q^+}
Assume 
\(
\lambda \ge \frac{8(C_2+C_0)^2}{C_1}+C_3.
\)
Let 
\(
\binom{u}{0}\in L_D \cap \bigl(L_0 + E_\lambda^s(+\infty)\bigr).
\)
Then there exists $v\in\R^n$ such that
\[
\binom{u}{0}
=
\binom{-M_\lambda v + u}{-v}
+
\binom{M_\lambda v}{v},
\]
with the first vector in $L_0$ and the second in $E_\lambda^s(+\infty)$.  
Hence
\[
\begin{aligned}
Q\!\left(L_D, L_0, E_\lambda^s(+\infty)\right)
\!\left(\binom{u}{0},\binom{u}{0}\right)
&=
\omega\!\left(
\binom{-M_\lambda v + u}{-v},
\binom{M_\lambda v}{v}
\right) \\
&=
\langle u, v\rangle
=
\langle M_\lambda v, v\rangle - \langle M_\lambda v - u, v\rangle.
\end{aligned}
\]
Since $\binom{M_\lambda v - u}{v}\in L_0$, Lemma~\ref{thm:epressionQ+} applied to $(u',v')=(M_\lambda v-u,v)$ yields
\[
Q\!\left(L_D,L_0,E_\lambda^s(+\infty)\right)\le 0,
\]
so
\[
\coiMor\bigl(Q(L_D,L_0,E_\lambda^s(+\infty))\bigr)=0.
\]
A symmetric argument for $E_\lambda^u(-\infty)$ gives
\[
\coiMor\bigl(Q(E_\lambda^{u}(-\infty), L_0, L_D)\bigr)=0.
\]
\end{rem}

\begin{cor}\label{cor:H2imply0} 
If {\rm(H2)} holds, then
\begin{equation}\label{corf:H20 }
\iota\bigl(E_0^{u}(-\infty),\,E_0^{s}(+\infty);\, L_D \bigr)=0.
\end{equation}
\end{cor}

\begin{proof}
Under {\rm(H2)},
\[
E_0^u(-\infty)=\left\{\binom{N_0u}{u}:u\in\R^n\right\},
\qquad
E_0^s(+\infty)=\left\{\binom{M_0u}{u}:u\in\R^n\right\},
\]
with $N_0>0$, $M_0<0$ by Lemma~\ref{thm:M0N0}.  
Using \eqref{eq:expression of Q} with $\lambda=0$,
\[
Q\!\left(E_0^{u}(-\infty),E_0^{s}(+\infty),L_D\right)
\Bigl(\binom{N_0u}{u},\binom{(N_0-M_0)u}{0}\Bigr)
=
\langle (M_0 - N_0)u,\,u\rangle,
\]
so the associated quadratic form is represented by $M_0-N_0$, which is negative definite.  
Hence
\[
\coiMor\!\left(Q(E_0^{u}(-\infty),E_0^{s}(+\infty);L_D)\right)=0.
\]
Since $E_0^{u}(-\infty)\cap L_D=\{0\}$, the triple-index formula implies
\[
\iota\!\left(E_0^{u}(-\infty),E_0^{s}(+\infty);L_D\right)
=
\coiMor\!\left(Q(E_0^{u}(-\infty),E_0^{s}(+\infty);L_D)\right)
=0.
\]
\end{proof}


\section{Proof of the main results}
In this section we prove the results stated in Section~\ref{sec:introduction}.

\subsection{Proof of Theorem~\ref{thm:morse-index-formula}}

Set
\[
R_\lambda := R + \lambda \Id,
\qquad
\widehat \lambda := \frac{8C_2^2}{C_1} + C_3.
\]
By construction and Lemmas~\ref{lem:matrix positive condition} and~\ref{thm:M0N0}, for
$\lambda\ge\widehat\lambda$ all asymptotic matrices $K_\lambda$ are positive definite and the
corresponding stable/unstable spaces are transversal.

By Remark~\ref{rmk:spectral-flow-morse-indices-difference},
\[
\iMor(u) = \Sf{\mathcal A_\lambda ; \lambda\in[0,\widehat\lambda]}.
\]
Applying the splitting formula~\eqref{eq:index-formula-1-of-2} gives
\begin{multline}\label{eq:morse index =maslov index:real line}
\iMor(u)
=
\iCLM\bigl(E^s_{\widehat{\lambda}}(\tau),E^u_{\widehat{\lambda}}(-\tau);\tau\in\R^+\bigr)
-
\iCLM\bigl(E^s_{0}(\tau),E^u_{0}(-\tau);\tau\in\R^+\bigr) \\
-
\iCLM\bigl(E^s_\lambda(+\infty),E^u_\lambda(-\infty);\lambda\in[0,\widehat{\lambda}]\bigr).
\end{multline}

\subsubsection*{Step 1: Asymptotic contribution at $\pm\infty$}

By Lemma~\ref{lem: maslov triple index (pair)},
\[
\iCLM\!\left(E^s_\lambda(+\infty),E^u_\lambda(-\infty);\lambda\in[0,\widehat{\lambda}]\right)
=
\iota\!\left(E_{\widehat\lambda}^{u}(-\infty), E_{\widehat\lambda}^{s}(+\infty); L_D \right)
-
\iota\!\left(E_{0}^{u}(-\infty), E_{0}^{s}(+\infty); L_D \right).
\]
For $\lambda=\widehat\lambda$, Lemma~\ref{thm:M0N0} and~\eqref{eq:expression of Q} imply that
$E_{\widehat\lambda}^{u}(-\infty)$ and $E_{\widehat\lambda}^{s}(+\infty)$ are graphs of definite
matrices (one positive, one negative) and form, together with $L_D$, a transverse triple. Hence
\[
\iota\!\left(E_{\widehat\lambda}^{u}(-\infty), E_{\widehat\lambda}^{s}(+\infty); L_D \right)=0,
\]
so
\[
\iCLM\!\left(E^s_\lambda(+\infty),E^u_\lambda(-\infty);\lambda\in[0,\widehat{\lambda}]\right)
=
-\iota\!\left(E_{0}^{u}(-\infty), E_{0}^{s}(+\infty); L_D \right).
\]

\subsubsection*{Step 2: Contribution along the real line at large $\lambda$}

By Proposition~\ref{lem:index=0 at infinity:real line}, for $\lambda=\widehat\lambda$ the stable
and unstable spaces satisfy
\[
E_{\widehat\lambda}^u(-\tau)\cap E_{\widehat\lambda}^s(\tau)=\{0\}
\quad \forall\,\tau>0,
\]
so the path
\(
\tau \mapsto (E^s_{\widehat\lambda}(\tau),E^u_{\widehat\lambda}(-\tau))
\)
has no crossings and
\[
\iCLM\left(E^s_{\widehat{\lambda}}(\tau),E^u_{\widehat{\lambda}}(-\tau);\tau\in\R^+\right)=0.
\]

\subsubsection*{Step 3: Conclusion}

Substituting into~\eqref{eq:morse index =maslov index:real line} yields
\[
\iMor(u)
=
-\,\iCLM\left(E^s_{0}(\tau),E^u_{0}(-\tau);\tau\in\R^+\right)
+
\iota\bigl(E_{0}^{u}(-\infty), E_{0}^{s}(+\infty); L_D \bigr).
\]
By Definition~\ref{def:h-clinic index},
\[
\igeo(u)
:=
-\,\iCLM\left(E^s_{0}(\tau),E^u_{0}(-\tau);\tau\in\R^+\right),
\]
whence
\[
\iMor(u)
=
\igeo(u)
+
\iota\bigl(E_{0}^{u}(-\infty), E_{0}^{s}(+\infty); L_D \bigr),
\]
which is Theorem~\ref{thm:morse-index-formula}. \qed

\subsection{Proof of Theorem~\ref{thm:morse index formula for halfclinic}}

We first treat the future halfclinic case, i.e.\ Equation~\eqref{eq:Morse Maslov+}.  
Under our assumptions,
\[
\Sf{\mathcal A_\lambda^+;\lambda\in[0,\widehat\lambda]}
=\iMor(u,L_0,+).
\]
Using~\eqref{eq:index-formula-1-of-2}, we obtain
\begin{align}\label{morse index+1.1}
\iMor(u,L_0,+)
&=
\iCLM\bigl(E_{\widehat \lambda}^{s}(\tau),L_0;\tau\in\R^+\bigr)
-
\iCLM\bigl(E_0^{s}(\tau),L_0;\tau\in\R^+\bigr) \\
&\quad -\,
\iCLM\bigl(E_\lambda^{s}(+\infty),L_0;\lambda\in[0,\widehat\lambda]\bigr).
\end{align}

\subsubsection*{Step 1: Large-$\lambda$ terms}

Now choose
\[
\widehat\lambda = \frac{8(C_2+C_0)^2}{C_1} + C_3,
\]
so Lemma~\ref{lem:matrx define*} applies at $+\infty$ and (H2) holds for the asymptotic system.
By Lemma~\ref{lem:index=0 at infinity:half line},
\[
E_{\widehat{\lambda}}^{s}(\tau)\cap L_0 = \{0\}
\quad\forall\,\tau>0,
\]
hence
\[
\iCLM\bigl(E_{\widehat{\lambda}}^{s}(\tau),L_0;\tau\in\R^+\bigr)=0.
\]

Moreover, $E_\lambda^{s}(+\infty)\cap L_D=\{0\}$ for all $\lambda\ge0$ (transversality at infinity),
so
\[
\iCLM\bigl(E_\lambda^{s}(+\infty),L_D;\lambda\in[0,\widehat{\lambda}]\bigr)=0.
\]

\subsubsection*{Step 2: Triple index term}

From Remark~\ref{rem:epression Q^+}, for 
\(
\lambda \ge \frac{8(C_2+C_0)^2}{C_1}+C_3
\)
the quadratic form
\(
Q(L_D,L_0,E_\lambda^{s}(+\infty))
\)
satisfies
\[
m^+\!\left(Q(L_D,L_0,E_\lambda^{s}(+\infty))\right)=0.
\]
Using~\eqref{morse index+1.1}, Definition~\ref{def:hormander index},
Equation~\eqref{eq:the Hormander index computed by triple index}, and~\eqref{eq:trip1}, we obtain
\begin{align}
\iMor(u,L_0,+)
&=
\iota^+_{L_0}(u)
-
\iCLM\bigl(E_\lambda^{s}(+\infty),L_0;\lambda\in[0,\widehat{\lambda}]\bigr) \\
&=
\iota^+_{L_0}(u)
-
\bigl(
\iCLM(E_\lambda^{s}(+\infty),L_0) 
-
\iCLM(E_\lambda^{s}(+\infty),L_D)
\bigr) \\
&=
\iota^+_{L_0}(u)
-
s(L_D,L_0;E_0^{s}(+\infty),E_{\widehat{\lambda}}^{s}(+\infty)) \\
&=
\iota^+_{L_0}(u)
-
\iota(L_D,L_0,E_{\widehat{\lambda}}^{s}(+\infty))
+
\iota(L_D,L_0,E_0^{s}(+\infty)).
\end{align}

As in Remark~\ref{rem:epression Q^+}, the triple
\(
(L_D,L_0,E_{\widehat{\lambda}}^{s}(+\infty))
\)
is transverse and the associated quadratic form has no positive eigenvalues, so
\[
\iota(L_D,L_0,E_{\widehat{\lambda}}^{s}(+\infty)) = 0.
\]
Hence
\[
\iMor(u,L_0,+)
=
\iota^+_{L_0}(u)
+
\iota\bigl(L_D,L_0,E_0^{s}(+\infty)\bigr),
\]
which is~\eqref{eq:Morse Maslov+}.

\subsubsection*{Step 3: Past halfclinic}

The past halfclinic case follows by the same argument, applied to the unstable subspace and the
path $\tau\mapsto E^u(-\tau)$. Using~\eqref{eq:index-formula-1-of-2},
Definition~\ref{def:hormander index} and~\eqref{eq:the Hormander index computed by triple index}, one
obtains
\[
\iMor(u,L_0,-)
=
\iota^-_{L_0}(u)
+
\iota\bigl(E_0^{u}(-\infty),L_0;L_D\bigr),
\]
which is~\eqref{eq:Morse Maslov -}. \qed

\subsection{Proof of Corollary~\ref{cor:morse index between general boundary and Dirichlet boundary}}

We first prove~\eqref{eq:+morse index between general boundary and Dirichlet boundary}.  
Using~\eqref{eq:Morse Maslov+}, Definition~\ref{def:hormander index} and
\eqref{eq:the Hormander index computed by triple index},
\begin{align}
\iMor(u,L_0,+)-\iMor(u,L_D,+ )
&=  \iota_{L_0}^+(u)+\iota\big( L_D,L_0,E^s_0(+\infty) \big)-\iota^+_{L_D}(u)\\
&= -\Big(  \iCLM\big(E^{s}(t), L_0 ;t\in\R^+\big)-    \iCLM\big(E^{s}(t), L_D ;t\in\R^+\big) \Big)
\\[-2mm]
&\qquad +\iota\big( L_D,L_0,E^s_0(+\infty) \big) \\[1mm]
&= -s\big( L_D,L_0,E^s_0(0),E^s_0(+\infty) \big)+\iota\big( L_D,L_0,E^s_0(+\infty) \big)\\
&= -\iota\big( L_D,L_0,E^s_0(+\infty) \big)+\iota\big( L_D,L_0,E^s_0(0) \big)\\
&\qquad +\iota\big( L_D,L_0,E^s_0(+\infty) \big)\\
&= \iota\big( L_D,L_0,E^s_0(0) \big),
\end{align}
which is~\eqref{eq:+morse index between general boundary and Dirichlet boundary}.

For~\eqref{eq:-morse index between general boundary and Dirichlet boundary}, we use
\eqref{eq:Morse Maslov -} and again Definition~\ref{def:hormander index} and
\eqref{eq:the Hormander index computed by triple index}:
\begin{align}
\iMor(u,L_0,-)-\iMor(u,L_D,- )  
&=\iota^-_{ L_0 }(u)+\iota\big(E_0^{u}(-\infty), L_0 ; L_D \big)-\iota^-_{ L_D }(u)\\
&=\iCLM\big( L_D ,E^{u}(-t);t\in\R^+\big)-\iCLM\big( L_0 ,E^{u}(-t);t\in\R^+\big)
\\[-2mm]
&\qquad +\iota\big(E_0^{u}(-\infty), L_0 ; L_D \big)\\[1mm]
&=s\big(E^{u}(0),E^{u}(-\infty); L_0 , L_D \big)+\iota\big(E_0^{u}(-\infty), L_0 ; L_D \big)\\
&=\iota\big(E^{u}(0), L_0 , L_D \big)-\iota\big(E^{u}(-\infty), L_0 ; L_D \big)
\\[-2mm]
&\qquad +\iota\big(E^{u}(-\infty), L_0 ; L_D \big)\\[1mm]
&=\iota\big(E^{u}(0), L_0 , L_D \big),
\end{align}
which is~\eqref{eq:-morse index between general boundary and Dirichlet boundary}. \qed

%
%
%

\section{Some classical examples}

In this section we illustrate the abstract Morse index formulas obtained above on a few
standard one- and multi-dimensional models admitting heteroclinic or halfclinic orbits.

\begin{itemize}
    \item In Subsection~\ref{pendulum} we treat the classical \emph{mathematical pendulum} and
    compute the Morse index of its heteroclinic and halfclinic trajectories.

    \item In Subsection~\ref{nagumo} we consider the scalar Nagumo equation and its
    heteroclinic and halfclinic solutions.

    \item In Subsection~\ref{reaction-diffusion} we study a coupled reaction--diffusion system in
    $\R^4$, showing how the method extends to genuinely higher-dimensional settings.
\end{itemize}

In each case we check that the structural assumptions (L1), (L2), (H1), (H2) are satisfied and
compute the associated Morse and Maslov indices explicitly.

%
%

\subsection{Mathematical pendulum}\label{pendulum}

Consider the one-dimensional pendulum equation
\begin{equation}\label{eq:pendulum}
    \ddot{\theta}(t) = V'(\theta(t)), \qquad
    V(\theta) = \frac{g}{l}\cos\theta,
\end{equation}
with Lagrangian
\[
    L(\theta,\dot\theta)
    =
    \frac12 |\dot\theta|^2 + V(\theta).
\]

\subsubsection*{Heteroclinic solutions}

It is classical (see e.g.\ \cite[Section~5]{BY11}) that \eqref{eq:pendulum} admits a heteroclinic
solution
\[
\widehat{\theta}(t)
=
4\arctan\!\left(\tanh\!\left(\tfrac12 t\sqrt{\tfrac{g}{l}}\right)\right),
\]
connecting $\theta=-\pi$ to $\theta=\pi$.

Linearisation along $\widehat{\theta}$ yields the Morse--Sturm equation
\begin{equation}\label{eq:p.m.s.}
\begin{cases}
- \phi''(t)
+
\displaystyle \frac{g}{l}\Bigl(1-2\,\sech^{2}\!\bigl( \sqrt{\tfrac{g}{l}}\, t \bigr)\Bigr)\phi(t)=0,\\[4pt]
\phi(t)\to 0 \quad\text{as } t\to\pm\infty,
\end{cases}
\end{equation}
with coefficients
\[
P(t)=1,\qquad Q(t)=0,\qquad
R(t)=\frac{g}{l}\Bigl(1-2\sech^{2}\!\bigl(\sqrt{\tfrac{g}{l}}t\bigr)\Bigr),\qquad
R_\pm=\frac{g}{l}.
\]
In particular, (L1), (L2), (H1), (H2) hold.

Setting $\Phi(t)=(\dot\phi(t),\phi(t))^T$, \eqref{eq:p.m.s.} is equivalent to
\begin{equation}\label{eq:pendulum.h.s.}
\begin{cases}
\dot{\Phi}(t)=J B(t)\Phi(t), \\[4pt]
\Phi(t)\to 0 \quad\text{as } t\to\pm\infty,
\end{cases}
\end{equation}
with
\[
B(t)=
\begin{pmatrix}
1 & 0 \\
0 & -\dfrac{g}{l}\Bigl(1-2\sech^{2}(\sqrt{\tfrac{g}{l}} t)\Bigr)
\end{pmatrix}.
\]

Let $E^{s}(\tau)$, $E^{u}(-\tau)$ denote the stable and unstable spaces of
\eqref{eq:pendulum.h.s.} at time $\tau$. By Corollary~\ref{thm:morse index=maslov index},
\[
\iMor(\widehat{\theta})
=
-\,\iCLM\bigl(E^{s}(\tau),E^{u}(-\tau);\tau\in\R^{+}\bigr).
\]

Using the explicit solution $\widehat{\theta}$, one checks that
\[
\binom{\ddot{\widehat{\theta}}(t)}{\dot{\widehat{\theta}}(t)}
\]
solves \eqref{eq:pendulum.h.s.}, and one obtains
\[
E^{s}(\tau)
=
\left\{
\binom{
-\sqrt{\tfrac{g}{l}}\,\tanh\!\left(\sqrt{\tfrac{g}{l}}\,\tau\right)v
}{v}
: v\in\R
\right\},
\qquad
E^{u}(\tau)
=
\left\{
\binom{
-\sqrt{\tfrac{g}{l}}\,\tanh\!\left(\sqrt{\tfrac{g}{l}}\,\tau\right)v
}{v}
: v\in\R
\right\},
\]
and
\[
E^{s}(+\infty)=
\left\{
\binom{-\sqrt{\tfrac{g}{l}}\,v}{v}:v\in\R
\right\},
\qquad
E^{u}(-\infty)=
\left\{
\binom{\sqrt{\tfrac{g}{l}}\,v}{v}:v\in\R
\right\}.
\]

Since $E^{s}(\tau)=E^{u}(\tau)$ for all $\tau$ and both are transversal to $L_D$,
Lemma~\ref{lem: maslov triple index (pair)} yields
\[
\iMor(\widehat{\theta})
=
\iota(E^{u}(0),E^{s}(0);L_{D})
-
\iota(E^{u}(-\infty),E^{s}(+\infty);L_{D}).
\]
A direct computation of the associated quadratic forms $Q$ (see~\eqref{eq:the form Q}) gives
\[
Q(E^{u}(0),E^{s}(0);L_{D})(v,v)=0,\qquad
Q(E^{u}(-\infty),E^{s}(+\infty);L_{D})(v,v)
=
-2\sqrt{\tfrac{g}{l}}\,v^{2},
\]
so both triple indices vanish and
\[
\iMor(\widehat{\theta})=0.
\]

\subsubsection*{Halfclinic solutions}

We view the same trajectory $\widehat{\theta}$ as a future halfclinic on $[0,+\infty)$:
\[
    \widehat{\theta}(t)
    =
    4 \arctan\!\left( \tanh\!\left(\frac{1}{2} t \sqrt{\frac{g}{l}}\right)\right),\qquad
    (\dot{\theta}(0),\theta(0))^T\in L_D,\quad
    \lim_{t\to +\infty}\theta(t)=\pi.
\]
Since $\widehat{\theta}(0)=0$, the boundary condition at $0$ is Dirichlet.

Linearisation gives the half-line problem
\begin{equation}\label{eq:p.m.s.future}
\begin{cases}
-\phi''(t)
+
\displaystyle\frac{g}{l}\Bigl(1-2 \sech^2\bigl( \sqrt{\tfrac{g}{l}}t  \bigr) \Bigr)\phi(t) = 0,\\[4pt]
(\dot\phi(0),\phi(0))^T\in L_D,\qquad 
\lim_{t\to +\infty}\phi(t)=0.
\end{cases}
\end{equation}
With $\Phi(t)=(\dot\phi(t),\phi(t))^T$ this is equivalent to
\begin{equation}\label{eq:pendulum.h.s.future}
\begin{cases}
\dot{\Phi}(t)=J B(t)\Phi(t),\\[4pt]
\Phi(0)\in L_D,\qquad \displaystyle\lim_{t\to+\infty}\Phi(t)=0,
\end{cases}
\end{equation}
with the same $B(t)$ as above. All assumptions (L1), (L2), (H1), (H2) hold.

By Theorem~\ref{thm:morse index formula for halfclinic} with $L_0=L_D$,
\[
\iMor(\widehat{\theta},L_D,+)
=
-\,\iCLM\bigl(E^s(\tau),L_D;\tau\in\R^+\bigr),
\]
since the Hörmander term vanishes when the two reference Lagrangians coincide.

From the above formulas,
\[
E^s(\tau)
=
\left\{
\binom{
-\sqrt{\tfrac{g}{l}}\tanh\!\left( \sqrt{\tfrac{g}{l}} \,\tau  \right) v
}{v}
: v\in\R
\right\},
\qquad
E^s(+\infty)
=
\left\{
\binom{-\sqrt{\tfrac{g}{l}}\,v}{v}: v\in\R
\right\}.
\]
For every $\tau>0$ we have $E^s(\tau)\cap L_D=\{0\}$, hence
\[
\iCLM\bigl(E^s(\tau),L_D;\tau\in\R^+\bigr)=0
\quad\Rightarrow\quad
\iMor(\widehat{\theta},L_D,+)=0.
\]

Now let
\[
L_0
=
\left\{
\binom{a v}{v} : v\in\R
\right\},\qquad a\in\R,
\]
be a general Lagrangian line, corresponding to the Robin condition $\dot\phi(0)=a\,\phi(0)$.

By Corollary~\ref{cor:morse index between general boundary and Dirichlet boundary},
\[
\iMor(\widehat{\theta},L_0,+)
=
\iota(L_D,L_0;E^s(0)).
\]
At $\tau=0$,
\[
E^s(0)
=
\left\{
\binom{0}{v}:v\in\R
\right\}
=L_N.
\]
For $V\in L_D\cap(L_0+E^s(0))$ we can write $V=(av,0)^T$ and decompose
\[
V
=
\binom{a v}{v} + \binom{0}{-v}
=:Y+Z,\quad
Y\in L_0,\ Z\in E^s(0).
\]
The quadratic form $Q$ of the triple $(L_D,L_0,E^s(0))$ satisfies
\[
Q(L_D,L_0;E^s(0))(V,V)
=
\omega(Y,Z)
=
- a v^2.
\]
Hence
\[
\iMor(\widehat{\theta},L_0,+)
=
\iota(L_D,L_0;E^s(0))
=
\begin{cases}
1,& a<0,\\[4pt]
0,& a\ge 0.
\end{cases}
\]

\begin{prop}
Let 
\[
\widehat{\theta}(t)
=
4 \arctan\!\left( \tanh\!\left(\tfrac{1}{2} t \sqrt{\tfrac{g}{l}}\right)\right)
\]
be the heteroclinic solution of~\eqref{eq:pendulum} connecting $-\pi$ to $\pi$. Then
\[
\iMor(\widehat{\theta}) = 0.
\]
Viewed on $[0,+\infty)$ as a future halfclinic with boundary
\[
(\dot{\theta}(0),\theta(0))^T\in L_0,\qquad
L_0
=
\left\{
\binom{a v}{v}: v\in\R
\right\},
\]
the Morse index is
\[
\iMor(\widehat{\theta},L_0,+)
=
\begin{cases}
1, & a<0,\\[4pt]
0, & a\ge 0.
\end{cases}
\]
\end{prop}

\subsection{Nagumo equation}\label{nagumo}

Consider the scalar Nagumo equation
\begin{equation}\label{eq:nagumo eq.}
u_t = u_{xx} + u(u-a)(1-u), \qquad -1 \le a \le 1,
\end{equation}
and, for $a=\frac12$, the steady equation
\begin{equation}\label{eq:nagumo steady equation}
u_{xx} + u\left(u-\tfrac12\right)(1-u) = 0.
\end{equation}

\subsubsection*{Heteroclinic solutions}

For $a=\frac12$ there is a monotone heteroclinic
\[
\widehat{u}(x)
=
\frac12 + \frac12 \tanh\!\left( \frac{\sqrt{2}}{4}x\right),
\qquad x\in\R,
\]
connecting $0$ to $1$ (see \cite{KT83}).

Linearisation of \eqref{eq:nagumo steady equation} at $\widehat{u}$ gives
\begin{align}\label{eq:1.r.d.eq.m.s}
\begin{cases}
- w''(x)
+
\Bigl(
\frac{3}{4}\tanh^2\!\left(\frac{\sqrt{2}}{4}x\right)
-\frac{1}{4}
\Bigr) w(x)
= 0, \quad x\in\R,\\[4pt]
w(x)\to 0 \quad \text{as } x\to\pm\infty,
\end{cases}
\end{align}
with
\[
P(x)=1,\quad Q(x)=0,\quad
R(x)=\frac{3}{4}\tanh^2\!\left(\frac{\sqrt{2}}{4}x\right)-\frac{1}{4},\quad
R_\pm=\frac12.
\]
Thus (L1), (L2), (H1), (H2) hold.

Setting $W(x)=(\dot w(x),w(x))^T$, we obtain the Hamiltonian system
\begin{align}\label{eq:Nagumo h.s.eq.}
\begin{cases}
\dot{W}(x) = J B(x) W(x),\\[4pt]
W(x)\to 0 \quad \text{as } x\to\pm\infty,
\end{cases}
\end{align}
where
\[
B(x)=
\begin{pmatrix}
1 & 0\\[2pt]
0 & \dfrac{3}{4}\tanh^2\!\left(\dfrac{\sqrt{2}}{4}x\right) - \dfrac{1}{4}
\end{pmatrix}.
\]

By Corollary~\ref{thm:morse index=maslov index},
\[
\iMor(\widehat{u})
=
-\,\iCLM\bigl(E^s(\tau),E^u(-\tau); \tau\in\R^+\bigr),
\]
where $\iMor(\widehat{u})$ is the Morse index of
\[
\mathcal A
=
- \frac{d^2}{dx^2} 
+ \frac{3}{4}\tanh^2\!\left(\frac{\sqrt{2}}{4}x\right) - \frac{1}{4}.
\]

By translation invariance, $w=\widehat{u}'$ solves $\mathcal A w=0$. Up to a constant factor,
\[
W(x)
=
\binom{\dot w(x)}{w(x)}
=
\begin{pmatrix}
-\tanh\!\left(\dfrac{\sqrt{2}}{4}x\right)\sech^{2}\!\left(\dfrac{\sqrt{2}}{4}x\right)\\[4pt]
\sqrt{2}\,\sech^{2}\!\left(\dfrac{\sqrt{2}}{4}x\right)
\end{pmatrix},
\]
and $W(x)\to 0$ as $x\to\pm\infty$. Hence for each $\tau$,
\[
E^s(\tau)
=
\left\{
\binom{-\tanh\!\left(\tfrac{\sqrt{2}}{4}\tau\right)v}{\sqrt{2}\,v}
: v\in\R
\right\}
=
E^u(\tau),
\]
and
\[
E^s(+\infty)
=
\left\{
\binom{-v}{\sqrt{2}\,v}:v\in\R
\right\},
\qquad
E^u(-\infty)
=
\left\{
\binom{v}{\sqrt{2}\,v}:v\in\R
\right\}.
\]

For all $\tau>0$ we have $E^s(\tau)\cap L_D = E^u(-\tau)\cap L_D = \{0\}$, so
\[
\iMor(\widehat{u})
=
\iota\bigl(E^u(0),E^s(0);L_D\bigr)
-
\iota\bigl(E^u(-\infty),E^s(+\infty);L_D\bigr).
\]
At $\tau=0$,
\[
E^s(0)=E^u(0)
=
\left\{
\binom{0}{\sqrt{2}\,v}:v\in\R
\right\}=L_N,
\]
and one finds
\[
Q(E^u(0),E^s(0);L_D)=0.
\]
At infinity,
\[
E^u(-\infty)
=
\left\{\binom{\tfrac{1}{\sqrt{2}}q}{q}:q\in\R\right\},
\quad
E^s(+\infty)
=
\left\{\binom{-\tfrac{1}{\sqrt{2}}q}{q}:q\in\R\right\},
\]
and
\[
Q(E^u(-\infty),E^s(+\infty);L_D)(v,v)
=
-2\sqrt{2}\,v^2.
\]
Thus both triple indices vanish and
\[
\iMor(\widehat{u})=0.
\]

\subsubsection*{Halfclinic solutions of the Nagumo equation}

Consider now the halfclinic problem
\begin{align}\label{eq:nagumo bdv}
\begin{cases}
u_{xx} + u\bigl(u-\tfrac{1}{2}\bigr)(1-u) = 0,\\[4pt]
(\dot u(0),u(0))^T\in L_0,\qquad \displaystyle\lim_{x\to+\infty}u(x)=1,
\end{cases}
\end{align}
with
\[
L_0 = \left\{ \binom{v}{2\sqrt{2}\,v} : v\in\R \right\}.
\]
Then
\[
\widehat{u}(x)
=
\frac{1}{2}
+
\frac{1}{2}\tanh\!\left(\frac{\sqrt{2}}{4}x\right),\qquad x\ge0,
\]
solves \eqref{eq:nagumo bdv}. Linearisation on $\R^+$ gives
\begin{align}\label{eq:nagumo m.s.halfclinic}
\begin{cases}
- w''(x)
+
\Bigl(
\frac{3}{4}\tanh^2\!\left(\dfrac{\sqrt{2}}{4}x\right)
-\frac{1}{4}
\Bigr)w(x)
=0,\\[4pt]
(\dot w(0),w(0))^T\in L_0,\qquad
\displaystyle\lim_{x\to+\infty}w(x)=0.
\end{cases}
\end{align}
As before, (L1), (L2), (H1), (H2) hold.

Let $W(x)=(\dot w(x),w(x))^T$. Then \eqref{eq:nagumo m.s.halfclinic} is equivalent to
\begin{align}\label{eq:Nagumo h.s.eq. half}
\begin{cases}
\dot W(x) = J B(x) W(x),\qquad x\ge0,\\[4pt]
W(0)\in L_0,\qquad \displaystyle\lim_{x\to+\infty}W(x)=0,
\end{cases}
\end{align}
with the same $B(x)$ as in~\eqref{eq:Nagumo h.s.eq.}. The stable spaces $E^s(\tau)$ are the same as
above.

For a general Lagrangian line
\[
L_1 = \left\{\binom{a v}{v}:v\in\R\right\},\qquad a\in\R,
\]
Theorem~\ref{thm:morse index formula for halfclinic} gives
\[
\iMor(\widehat{u},L_1,+)
=
-\,\iCLM\bigl(E^s(\tau),L_1;\tau\in\R^+\bigr)
+
\iota(L_D,L_1;E^s(+\infty)).
\]
Using Corollary~\ref{cor:morse index between general boundary and Dirichlet boundary} and the fact that
$\iCLM(E^s(\tau),L_D;\tau\in\R^+)=0$, we obtain
\[
\iMor(\widehat{u},L_1,+)
=
\iota(L_D,L_1;E^s(0)).
\]

At $\tau=0$,
\[
E^s(0)
=
\left\{\binom{0}{\sqrt{2}\,v}:v\in\R\right\}
=L_N.
\]
As before, any $V\in L_D\cap(L_1+E^s(0))$ can be written as $V=(a v,0)^T$ and decomposed as
\[
V
=
\binom{a v}{v} + \binom{0}{-v},
\]
with the quadratic form
\[
Q(L_D,L_1;E^s(0))(V,V)
=
\omega\!\left(\binom{a v}{v},\binom{0}{-v}\right)
=
- a v^2.
\]
Hence
\[
\iMor(\widehat{u},L_1,+)
=
\begin{cases}
1, & a<0,\\[4pt]
0, & a\ge 0.
\end{cases}
\]

\begin{prop}
Let 
\[
\widehat{u}(x)
=
\frac{1}{2}
+
\frac{1}{2}\tanh\!\left(\frac{\sqrt{2}}{4}x\right)
\]
be the heteroclinic solution of~\eqref{eq:nagumo steady equation}. Then
\[
\iMor(\widehat{u})=0.
\]
Viewed on $[0,+\infty)$ as a future halfclinic solution with boundary
\[
(\dot u(0),u(0))^T\in L_1,\qquad
L_1
=
\left\{\binom{a v}{v}:v\in\R\right\},
\]
the Morse index satisfies
\[
\iMor(\widehat{u},L_1,+)
=
\begin{cases}
1, & a<0,\\[4pt]
0, & a\ge 0.
\end{cases}
\]
\end{prop}

\subsection{Coupled reaction--diffusion system in dimension four}\label{reaction-diffusion}

Consider the coupled reaction--diffusion system
\begin{align}\label{eq:r.d.eq}
\begin{cases}
u_t = u_{xx} + u\bigl(u-\tfrac{1}{2}\bigr)(1-u) + c(u-v),\\[4pt]
v_t = v_{xx} + v\bigl(v-\tfrac{1}{2}\bigr)(1-v) + c(v-u),
\end{cases}
\end{align}
with $c\in\R$ and $c<\frac14$. The steady equation can be written as
\begin{equation}\label{eq:2ndhs}
w''(x) = \nabla V(w(x)), \qquad w=(u,v)^T,
\end{equation}
where
\[
V(w)
=
\frac{1}{4}(u^4+v^4)
-\frac{1}{2}(u^3+v^3)
+\frac{1}{4}(u^2+v^2)
-\frac{c}{2}(u-v)^2,
\]
and
\[
L(w,w')
=
\frac{1}{2}|w'|^2 + V(w)
\]
is the associated Lagrangian.

Equation~\eqref{eq:2ndhs} admits the heteroclinic solution
\[
\widehat{w}(x)
=
\begin{pmatrix}
\widehat{u}(x)\\[2pt]
\widehat{u}(x)
\end{pmatrix}
=
\begin{pmatrix}
\dfrac{1}{2}+\dfrac{1}{2}\tanh\!\left(\dfrac{\sqrt{2}}{4}x\right)\\[4pt]
\dfrac{1}{2}+\dfrac{1}{2}\tanh\!\left(\dfrac{\sqrt{2}}{4}x\right)
\end{pmatrix},
\]
connecting $(0,0)^T$ to $(1,1)^T$. For $c<\frac14$ both rest points are hyperbolic.

Linearisation at $\widehat{w}$ leads to
\begin{equation}\label{eq:r.d.eq.ms}
\begin{cases}
- w''(x)
+
D^2V(\widehat{w}(x))\,w(x)
=0, \quad x\in\R,\\[4pt]
w(x)\to 0 \quad\text{as } x\to\pm\infty,
\end{cases}
\end{equation}
with
\[
D^2V(\widehat{w}(x))
=
\begin{pmatrix}
\frac{3}{4}\tanh^2\!\left(\dfrac{\sqrt{2}}{4}x\right) - \frac{1}{4} - c & c\\[4pt]
c & \frac{3}{4}\tanh^2\!\left(\dfrac{\sqrt{2}}{4}x\right) - \frac{1}{4} - c
\end{pmatrix},
\]
and
\[
R_\pm =
\begin{pmatrix}
\frac{1}{2}-c & c\\[4pt]
c & \frac{1}{2}-c
\end{pmatrix}.
\]
Thus (L1), (L2), (H1), (H2) are satisfied.

Let $W(x)=(\dot w(x),w(x))^T$. Then \eqref{eq:r.d.eq.ms} is equivalent to
\begin{align}\label{eq:coupled.r.d.eq.h.s.}
\begin{cases}
\dot{W}(x) = J B(x) W(x),\\[4pt]
W(x)\to 0 \quad\text{as } x\to\pm\infty,
\end{cases}
\end{align}
with
\[
B(x)
=
\begin{pmatrix}
\Id & 0\\[2pt]
0 & -D^2V(\widehat{w}(x))
\end{pmatrix}.
\]
By Corollary~\ref{thm:morse index=maslov index},
\[
\iMor(\widehat{w})
=
-\iCLM\bigl(E^s(\tau),E^u(-\tau);\tau\in\R^+\bigr),
\]
where $\iMor(\widehat{w})$ is the Morse index of
\[
\mathscr{A}
=
- \frac{d^2}{dx^2}
+
D^2V(\widehat{w}(x))
\]
on $W^{2,2}(\R,\R^2)$.

\subsubsection*{Spectral decomposition}

To analyse the spectrum of $\mathscr{A}$, introduce
\[
T
=
\begin{pmatrix}
1 & -1\\[2pt]
1 & 1
\end{pmatrix},
\qquad
T^{-1}
=
\frac{1}{2}
\begin{pmatrix}
1 & 1\\[2pt]
-1 & 1
\end{pmatrix},
\]
and set $z = T^{-1}w$. One checks that
\[
\widehat{\mathscr{A}}
:=
T^{-1}\mathscr{A}T
=
- \frac{d^2}{dx^2}
+
\begin{pmatrix}
\displaystyle\frac{3}{4}\tanh^2\!\left(\dfrac{\sqrt{2}}{4}x\right) - \frac{1}{4}
& 0\\[8pt]
0 &
\displaystyle\frac{3}{4}\tanh^2\!\left(\dfrac{\sqrt{2}}{4}x\right) - \frac{1}{4} -2c
\end{pmatrix}.
\]
Thus $\mathscr{A}$ and $\widehat{\mathscr{A}}$ are unitarily equivalent and
\[
\widehat{\mathscr{A}} = \mathscr{L}_1 \oplus \mathscr{L}_2,
\]
with
\begin{align*}
\mathscr{L}_1
&=
- \frac{d^2}{dx^2}
+
\Bigl(\frac{3}{4}\tanh^2\!\left(\frac{\sqrt{2}}{4}x\right) - \frac{1}{4}\Bigr),\\[2pt]
\mathscr{L}_2
&=
- \frac{d^2}{dx^2}
+
\Bigl(\frac{3}{4}\tanh^2\!\left(\frac{\sqrt{2}}{4}x\right) - \frac{1}{4} -2c\Bigr).
\end{align*}

The scalar spectral problem
\[
-\phi''(x)
+
\Bigl(\frac{3}{4}\tanh^2\!\left(\frac{\sqrt{2}}{4}x\right) - \frac{1}{4}\Bigr)\phi(x)
=
\lambda \phi(x)
\]
is of Pöschl--Teller type. The function
\[
\phi_0(x) = 1 - \tanh^2\!\left(\frac{\sqrt{2}}{4}x\right)
=
\sech^2\!\left(\frac{\sqrt{2}}{4}x\right)
\]
is an eigenfunction with $\lambda=0$, and the potential tends to $\frac12$ at infinity, so the
essential spectrum lies in $[\frac12,+\infty)$. A standard analysis shows that $\lambda=0$ is the
unique eigenvalue of $\mathscr{L}_1$.

Since
\[
\mathscr{L}_2 = \mathscr{L}_1 - 2c,
\]
the eigenvalues of $\mathscr{L}_2$ are of the form $\lambda_k - 2c$. As $\mathscr{L}_1$ has only
$\lambda_0=0$, $\mathscr{L}_2$ has the single eigenvalue $-2c$, with eigenfunction $\phi_0$.

Thus:
\begin{itemize}
\item If $c\le 0$, then $-2c\ge 0$ and both eigenvalues of $\widehat{\mathscr{A}}$ are non-negative.
\item If $0<c<\frac{1}{4}$, then $-2c<0$ and $\widehat{\mathscr{A}}$ has exactly one negative eigenvalue
$-2c$.
\end{itemize}

Therefore
\[
\iMor(\widehat{\mathscr{A}})
=
\begin{cases}
0, & c\le 0,\\[4pt]
1, & 0<c<\dfrac{1}{4}.
\end{cases}
\]
Since $\mathscr{A}$ and $\widehat{\mathscr{A}}$ are equivalent, they have the same Morse index, and
\[
\iMor(\widehat{w})
=
\begin{cases}
0, & c\le 0,\\[4pt]
1, & 0<c<\dfrac{1}{4}.
\end{cases}
\]

\begin{prop}
Let
\[
\widehat{w}(x)
=
\begin{pmatrix}
\dfrac{1}{2}+\dfrac{1}{2}\tanh\!\left(\dfrac{\sqrt{2}}{4}x\right)\\[4pt]
\dfrac{1}{2}+\dfrac{1}{2}\tanh\!\left(\dfrac{\sqrt{2}}{4}x\right)
\end{pmatrix}
\]
be the heteroclinic solution of~\eqref{eq:r.d.eq}. Then
\[
\iMor(\widehat{w})
=
\begin{cases}
0, & c\le 0,\\[4pt]
1, & 0<c<\dfrac{1}{4}.
\end{cases}
\]
\end{prop}


\appendix

\section{Maslov, H\"ormander and triple index}\label{sec:maslov}

This appendix collects the basic definitions and properties of the Maslov–CLM index, the
H\"ormander index and the triple index used throughout the paper.  
Our main references are \cite{CLM94,RS93,HP17,ZWZ18}.

\subsection{The Cappell–Lee–Miller index}\label{subsec:Maslov}

Let $(\R^{2n},\omega)$ be the standard symplectic space and $\Lagr(n)$ its Lagrangian
Grassmannian. For $a<b$ we denote by $\P([a,b];\R^{2n})$ the space of ordered pairs
\[
L:[a,b]\to\Lagr(n)\times\Lagr(n),\qquad
t\longmapsto L(t)=(L_1(t),L_2(t)),
\]
endowed with the compact–open topology.

Following Cappell–Lee–Miller \cite{CLM94}, we recall the Maslov index for pairs of Lagrangian
paths.

\begin{defn}
The \emph{Cappell–Lee–Miller index} (CLM index) is the unique map
\[
\iCLM:\P([a,b];\R^{2n})\to\Z,\qquad
L\longmapsto\iCLM\bigl(L(t);t\in[a,b]\bigr),
\]
satisfying axioms (I)–(VI) in \cite[Section~1]{CLM94}.
\end{defn}

Loosely speaking, for a pair $L=(L_1,L_2)$ the integer $\iCLM(L_1,L_2)$ counts, with signs and
multiplicities, the instants $t\in[a,b]$ such that $L_1(t)\cap L_2(t)\neq\{0\}$.

We will use repeatedly the following basic properties.

\begin{itemize}
\item[]\textbf{(Reversal)}
Let $L=(L_1,L_2)\in\P([a,b];\R^{2n})$ and define
\[
\widehat{L}(t)=(L_1(-t),L_2(-t)),\qquad t\in[-b,-a].
\]
Then
\[
\iCLM\bigl(\widehat{L}(t);t\in[-b,-a]\bigr)
=
-\iCLM\bigl(L(t);t\in[a,b]\bigr).
\]

\item[]\textbf{(Symplectic invariance)}
Let $L=(L_1,L_2)\in\P([a,b];\R^{2n})$ and
$\phi\in\mathscr C^0([a,b],\Sp(2n,\R))$. Then
\[
\iCLM\bigl(\phi(t)L_1(t),\phi(t)L_2(t);t\in[a,b]\bigr)
=
\iCLM\bigl(L_1(t),L_2(t);t\in[a,b]\bigr).
\]
\end{itemize}

Together with fixed-endpoint homotopy invariance, these properties are the main tools in the
applications to Hamiltonian systems and Morse–Maslov index theorems used in the paper.

\subsection{Triple index and H\"ormander index}\label{subsec:Hor_index}

We now recall the triple index and the H\"ormander index, and their relation.  
Our main reference is \cite{ZWZ18}.

Let $\alpha,\beta,\delta$ be isotropic subspaces of $(\R^{2n},\omega)$.  
Define the quadratic form
\begin{equation}\label{eq:the form Q}
Q=Q(\alpha,\beta;\delta):\alpha\cap(\beta+\delta)\to\R
\end{equation}
by
\[
Q(x_1,x_2)=\omega(y_1,z_2),
\]
where $x_j\in\alpha\cap(\beta+\delta)$ is decomposed as $x_j=y_j+z_j$ with $y_j\in\beta$,
$z_j\in\delta$ ($j=1,2$).  
This is well defined; see \cite[Lemma~3.3]{ZWZ18}.

If $\alpha,\beta,\delta$ are Lagrangian, then \cite[Lemma~3.3]{ZWZ18} shows
\begin{equation}\label{Qker}
\ker Q(\alpha,\beta;\delta)
=
\alpha\cap\beta+\alpha\cap\delta.
\end{equation}

\begin{defn}\label{lem:defi:triple index}
Let $\alpha,\beta,\kappa\in\Lagr(n)$.  
The \emph{triple index} of $(\alpha,\beta,\kappa)$ is
\begin{equation}\label{eq:the triple index}
\iota(\alpha,\beta,\kappa)
=
\iMor\bigl(Q(\alpha,\delta;\beta)\bigr)
+
\iMor\bigl(Q(\beta,\delta;\kappa)\bigr)
-
\iMor\bigl(Q(\alpha,\delta;\kappa)\bigr),
\end{equation}
where $\delta\in\Lagr(n)$ is any Lagrangian satisfying
\[
\delta\cap\alpha=\delta\cap\beta=\delta\cap\kappa=\{0\}.
\]
The right-hand side is independent of the choice of $\delta$ \cite{ZWZ18}.
\end{defn}

A more direct expression is given in \cite[Lemma~3.13]{ZWZ18}:

\begin{equation}\label{eq:trip1}
\iota(\alpha,\beta,\kappa)
=
m^{+}\bigl(Q(\alpha,\beta;\kappa)\bigr)
+
\dim(\alpha\cap\kappa)
-
\dim(\alpha\cap\beta\cap\kappa),
\end{equation}
where $m^+(\cdot)$ denotes the positive index of inertia.

\medskip

We now turn to the H\"ormander index, which measures the change of Maslov index when the
reference Lagrangian is varied.

Let $V_0,V_1,L_0,L_1\in\Lagr(n)$ and
\[
L\in\mathscr C^0([0,1],\Lagr(n)),\quad L(0)=L_0,\;L(1)=L_1,
\]
\[
V\in\mathscr C^0([0,1],\Lagr(n)),\quad V(0)=V_0,\;V(1)=V_1.
\]

\begin{defn}\label{def:hormander index} 
The \emph{H\"ormander index} is
\begin{align*}
s(L_0,L_1;V_0,V_1)
&=
\iCLM\bigl(V_1,L(t);t\in[0,1]\bigr)
-
\iCLM\bigl(V_0,L(t);t\in[0,1]\bigr)\\
&=
\iCLM\bigl(V(t),L_1;t\in[0,1]\bigr)
-
\iCLM\bigl(V(t),L_0;t\in[0,1]\bigr).
\end{align*}
\end{defn}

By homotopy invariance, this does not depend on the particular paths $L,V$ connecting the
endpoints; see \cite{RS93}.

Let now $\lambda_1,\lambda_2,\kappa_1,\kappa_2\in\Lagr(n)$.  
By \cite[Theorem~1.1]{ZWZ18},
\begin{equation}\label{eq:the Hormander index computed by triple index}
s(\lambda_1,\lambda_2;\kappa_1,\kappa_2)
=
\iota(\lambda_1,\lambda_2,\kappa_2)
-
\iota(\lambda_1,\lambda_2,\kappa_1)
=
\iota(\lambda_1,\kappa_1,\kappa_2)
-
\iota(\lambda_2,\kappa_1,\kappa_2).
\end{equation}

\begin{lem}[\cite{HWY18}]\label{lem:Maslov and triple index} 
Let $L_0,L\in\Lagr(n)$ and $l\in\mathscr C^0([0,1],\Lagr(n))$.  
Assume $l(t)\cap L=\{0\}$ for all $t\in[0,1]$. Then
\[
\iCLM\bigl(L_0,l(t);t\in[0,1]\bigr)
=
\iota\bigl(l(1),L_0;L\bigr)
-
\iota\bigl(l(0),L_0;L\bigr).
\]
\end{lem}

The triple index is symplectically invariant: for $\alpha,\beta,\kappa\in\Lagr(n)$ and
$\phi\in\Sp(2n,\R)$,
\[
\iota(\phi\alpha,\phi\beta,\phi\kappa)=\iota(\alpha,\beta,\kappa).
\]
Using this, Lemma~\ref{lem:Maslov and triple index} extends to pairs of Lagrangian paths.

\begin{lem}\label{lem: maslov triple index (pair)} 
Let $l_1,l_2\in\mathscr C^0([0,1],\Lagr(n))$ and let $L\in\Lagr(n)$ be such that
\[
l_1(t)\cap L = l_2(t)\cap L = \{0\} \qquad\forall\,t\in[0,1].
\]
Then
\[
\iCLM\bigl(l_1(t),l_2(t);t\in[0,1]\bigr)
=
\iota\bigl(l_2(1),l_1(1);L\bigr)
-
\iota\bigl(l_2(0),l_1(0);L\bigr).
\]
\end{lem}

\begin{proof}
Write $\R^{2n}=L\oplus JL$ and choose symmetric matrices $M(t),N(t)$ such that
\[
l_1(t)
=
\left\{\begin{pmatrix}M(t)v\\ v\end{pmatrix}:v\in\R^n\right\},
\quad
l_2(t)
=
\left\{\begin{pmatrix}N(t)v\\ v\end{pmatrix}:v\in\R^n\right\}.
\]
Define
\[
T(t)
=
\begin{pmatrix}
\Id & M(0)-M(t)\\[1pt]
0 & \Id
\end{pmatrix}\in\Sp(2n,\R).
\]
Then
\[
T(t)l_1(t)
=
\left\{\begin{pmatrix}M(0)v\\ v\end{pmatrix}:v\in\R^n\right\},
\quad
T(t)l_2(t)
=
\left\{\begin{pmatrix}(N(t)-M(t)+M(0))v\\ v\end{pmatrix}:v\in\R^n\right\}.
\]
By symplectic invariance of $\iCLM$ and Lemma~\ref{lem:Maslov and triple index},
\begin{align*}
\iCLM\bigl(l_1(t),l_2(t);t\in[0,1]\bigr)
&=
\iCLM\bigl(T(t)l_1(t),T(t)l_2(t);t\in[0,1]\bigr)\\
&=
\iota\bigl(T(1)l_2(1),T(1)l_1(1);L\bigr)
-
\iota\bigl(T(0)l_2(0),T(0)l_1(0);L\bigr).
\end{align*}
Since $T(0)=\Id$ and $T(1)$ is symplectic, the triple index is invariant under $T(1)$, and the
claim follows.
\end{proof}
\begin{rem}[Interpretation in the ODE examples]\label{rem:interpretation_examples}
For the Hamiltonian systems obtained by linearising the heteroclinic and 
halfclinic orbits (pendulum, Nagumo equation, and the coupled reaction–diffusion 
model), the abstract symplectic indices have the following concrete meanings.

\begin{enumerate}
\item 
The Maslov–CLM index of the pair $\bigl(E^{s}(\tau),E^{u}(-\tau)\bigr)$ counts, 
with sign, the times at which the stable and unstable directions of the 
linearised flow cease to be transverse.  
Such crossings correspond to zero eigenfunctions of the associated 
Sturm–Liouville operator satisfying the decay conditions at $\pm\infty$.

\item The H\"ormander index compares Maslov indices computed with respect to different boundary
Lagrangians (Dirichlet, Neumann, or general selfadjoint boundary conditions).  
This is the key tool in expressing differences of Morse indices for different boundary conditions.

\item The triple index allows one to express the CLM index in terms of algebraic data at the
endpoints, via the quadratic form $Q(\cdot,\cdot;\cdot)$ and formula \eqref{eq:trip1}.  
In the examples of Sections~\ref{pendulum}, \ref{nagumo} and \ref{reaction-diffusion}, this
reduces the Morse index computation to the definiteness of simple $1\times1$ or $2\times2$
matrices associated with the limiting Lagrangian subspaces
\[
E^u(-\infty),\quad E^s(+\infty),\quad L_D,\quad L_0.
\]
\end{enumerate}
\end{rem}

\section{Spectral flow for paths of selfadjoint Fredholm operators}\label{sec:spectral-flow}

Let $(\mathcal H,(\cdot,\cdot))$ be a real separable Hilbert space and denote by
$\cfsa(\mathcal H)$ the space of closed, selfadjoint Fredholm operators on $\mathcal H$,
equipped with the gap topology.  

For $T\in\cfsa(\mathcal H)$ and $a<b$ outside the spectrum of $T$, let
\[
P_{[a,b]}(T)
=
\frac{1}{2\pi i}\int_\gamma (\lambda I - T^{\mathbb C})^{-1}\,d\lambda,
\]
where $\gamma$ is the positively oriented circle centred at $\tfrac{a+b}{2}$ with radius
$\tfrac{b-a}{2}$.  
If $[a,b]$ contains only isolated eigenvalues of $T$ of finite multiplicity, then
\[
\im P_{[a,b]}(T)
=
E_{[a,b]}(T)
=
\bigoplus_{\lambda\in[a,b]}\ker(\lambda\Id-T).
\]

Let $\mathcal A:[a,b]\to\cfsa(\mathcal H)$ be continuous.  
By \cite[Proposition~2.10]{BLP05}, for each $t_0\in[a,b]$ there exist $a_{t_0}>0$ and a
neighbourhood $\mathscr N_{t_0}\subset\cfsa(\mathcal H)$ of $\mathcal A(t_0)$ such that
$\pm a_{t_0}\notin\sigma(T)$ for all $T\in\mathscr N_{t_0}$ and
\[
T\longmapsto P_{[-a_{t_0},a_{t_0}]}(T)
\]
is continuous on $\mathscr N_{t_0}$.  
In particular, $\dim E_{[-a_{t_0},a_{t_0}]}(T)$ is constant on $\mathscr N_{t_0}$.

Choosing a finite partition $a=t_0<t_1<\cdots<t_N=b$ so that
$\mathcal A([t_{i-1},t_i])\subset\mathscr N_{t_i}$ for suitable $a_i>0$, the dimensions
$\dim E_{[-a_i,a_i]}(\mathcal A_t)$ are constant on each $[t_{i-1},t_i]$.

\begin{defn}\label{def:spectral-flow-unb}
The \emph{spectral flow} of $\mathcal A$ on $[a,b]$ is
\[
\Sf{\mathcal A_\lambda;\lambda\in[a,b]}
=
\sum_{i=1}^N
\bigl(
  \dim E_{[0,a_i]}(\mathcal A_{t_i})
  -
  \dim E_{[0,a_i]}(\mathcal A_{t_{i-1}})
\bigr)\in\Z.
\]
\end{defn}

Informally, this is the net number of eigenvalues of $\mathcal A_\lambda$ crossing $0$ from
negative to positive as $\lambda$ increases from $a$ to $b$.

\medskip

\subsection*{Crossing forms}

Assume now that $\mathcal A$ is of class $\mathscr C^1$.  
Let $P_t$ be the orthogonal projection onto $\ker\mathcal A_t$.  
A point $t_0\in[a,b]$ is a \emph{crossing} if $\ker\mathcal A_{t_0}\neq\{0\}$.  
Define the \emph{crossing operator}
\[
\Gamma(\mathcal A,t_0)
:=
P_{t_0}\,\dot{\mathcal A}_{t_0}\,P_{t_0}
:
\ker\mathcal A_{t_0}\to\ker\mathcal A_{t_0}.
\]
A crossing is \emph{regular} if $\Gamma(\mathcal A,t_0)$ is non-degenerate.

If all crossings are regular, they are isolated. Let $\mathcal S$ be the set of crossings and
$\mathcal S_*=\mathcal S\cap(a,b)$.  
Denote by $\sgn(\Gamma)$ the signature of a selfadjoint operator $\Gamma$, i.e.
\[
\sgn(\Gamma)
=
\dim E_{+}(\Gamma)-\dim E_{-}(\Gamma).
\]
Then
\begin{equation}\label{eq:compute sf by crossing form}
\Sf{\mathcal A_t;\,t\in[a,b]}
=
\sum_{t_0\in\mathcal S_*}\sgn(\Gamma(\mathcal A,t_0))
-
\dim E_-(\Gamma(\mathcal A,a))
+
\dim E_+(\Gamma(\mathcal A,b)).
\end{equation}

\begin{rem}\label{rem:regular crossing}
Generic perturbations yield regular crossings: there exists $\varepsilon_0>0$ such that
\[
\Sf{\mathcal{A}_t;\,t\in[a,b]}
=
\Sf{\mathcal{A}_t+\varepsilon I;\,t\in[a,b]}
\qquad\forall\,\varepsilon\in[0,\varepsilon_0],
\]
and for almost all such $\varepsilon$ the path
$t\mapsto\mathcal A_t+\varepsilon I$ has only regular crossings; see \cite{CH07,HP17}.
\end{rem}

\subsection*{Positive curves}

\begin{defn}[\cite{HW18}]\label{def:positive curve}
A continuous path $\mathcal A:[a,b]\to\cfsa(\mathcal H)$ is a \emph{positive curve} if the set
\[
\{\lambda\in[a,b]\mid\ker\mathcal A_\lambda\neq0\}
\]
is finite and
\[
\Sf{\mathcal A_\lambda;\lambda\in[a,b]}
=
\sum_{a<\lambda\le b}\dim\ker\mathcal A_\lambda.
\]
That is, along a positive curve the spectral flow simply counts the total multiplicity of
eigenvalues crossing zero from negative to nonnegative as $\lambda$ increases.
\end{defn}


\section{Hyperbolicity}

In this section we collect several sufficient conditions ensuring the hyperbolicity of the Hamiltonian
matrix associated with a Sturm–Liouville system. We begin with a characterisation of the
hyperbolicity of $JB$ (with $B$ as below) in terms of the non-vanishing of a suitable determinant.

Consider the block matrix
\begin{equation}\label{matrix:B}
B
=
\begin{pmatrix}
P^{-1} & -P^{-1}Q\\[2pt]
-Q^T P^{-1} & Q^T P^{-1} Q - R
\end{pmatrix},
\end{equation}
where $P$ is assumed to be invertible.

\begin{lem}\label{lem:criteria_hypebolic}
Assume that $P$ is invertible. Then the Hamiltonian matrix $JB$ is hyperbolic (i.e.\ it has no purely
imaginary eigenvalues) if and only if
\[
\det\bigl(a^2P + i a(Q^T-Q) + R\bigr)\neq 0
\quad\text{for every } a\in\R.
\]
\end{lem}

\begin{proof}
The matrix $JB$ is hyperbolic if and only if $\det(JB + i a\Id)\neq 0$ for all $a\in\R$.
Using $J^{-1}=-J$ and $\det J=1$, we obtain
\[
\det(JB + i a\Id) = \det(B - i a J),
\]
up to the harmless change $a\mapsto -a$. A direct block computation gives
\begin{align*}
\det(B - i a J)
&=
\det\begin{pmatrix}
P^{-1} & -P^{-1}Q + i a\Id\\[2pt]
-Q^T P^{-1} - i a\Id & Q^T P^{-1}Q - R
\end{pmatrix}\\
&=
\det\!\left[
\begin{pmatrix}
\Id & 0\\[2pt]
(Q^T P^{-1} + i a\Id)P & \Id
\end{pmatrix}
\begin{pmatrix}
P^{-1} & -P^{-1}Q + i a\Id\\[2pt]
-Q^T P^{-1} - i a\Id & Q^T P^{-1}Q - R
\end{pmatrix}
\right]\\
&=
\det\begin{pmatrix}
P^{-1} & -P^{-1}Q + i a\Id\\[2pt]
0 & -i a Q + i a Q^T - a^2 P - R
\end{pmatrix}\\
&=
\det P^{-1}\,\det\bigl(-a^2P - R - i a(Q - Q^T)\bigr).
\end{align*}
Since $\det P^{-1}\neq 0$, we obtain
\[
\det(JB + i a\Id)\neq 0
\quad\Longleftrightarrow\quad
\det\bigl(a^2P + i a(Q^T-Q) + R\bigr)\neq 0,
\]
which proves the claim.
\end{proof}

The next result gives a convenient sufficient condition for hyperbolicity in terms of the positivity
of the block matrix built from the Sturm–Liouville coefficients.

\begin{cor}\label{cor:hyper_positive_definite}
If the block matrix
\[
\begin{pmatrix}
P & Q\\[2pt]
Q^T & R
\end{pmatrix}
\]
is positive definite, then $JB$ is hyperbolic.
\end{cor}

\begin{proof}
For each $a\in\R$ consider the matrix
\[
a^2P + i a(Q^T-Q) + R.
\]
We can rewrite it as
\[
a^2P + i a(Q^T-Q) + R
=
\begin{pmatrix} i a\Id & \Id \end{pmatrix}
\begin{pmatrix}
P & Q\\[2pt]
Q^T & R
\end{pmatrix}
\begin{pmatrix}
- i a\Id\\[2pt]
\Id
\end{pmatrix}.
\]
Since the block matrix in the middle is positive definite by assumption, the above expression is
positive definite (and in particular invertible) for every $a\in\R$.  
The claim then follows from Lemma~\ref{lem:criteria_hypebolic}.
\end{proof}

We next consider a family of Hamiltonian matrices obtained by adding a scalar multiple of the identity
to the potential $R$. The following lemma shows that, under a natural positivity assumption on $P$,
hyperbolicity at $\lambda=0$ propagates to all $\lambda\ge 0$.

\begin{lem}\label{thm:hyperbolic-at-starting-point}
Let $[0,1]\ni\lambda\longmapsto R_\lambda:=R+\lambda\Id\in\Sym(n)$ and let $B_\lambda$ be obtained
from $B$ in \eqref{matrix:B} by replacing $R$ with $R_\lambda$.  
Assume that $P\in\Sym^+(n)$ and that $JB_0$ is hyperbolic.  
Then $JB_\lambda$ is hyperbolic for all $\lambda\in\R^+$.
\end{lem}

\begin{proof}
For each $\lambda\ge 0$ define
\[
f_\lambda(a)
=
a^2P + i a(Q^T-Q) + R + \lambda\Id,
\qquad a\in\R.
\]
By Lemma~\ref{lem:criteria_hypebolic}, the hyperbolicity of $JB_0$ is equivalent to
\[
\det f_0(a)\neq 0 \quad\text{for all } a\in\R.
\]

Since $P>0$, for $|a|$ large the leading term $a^2P$ dominates and all eigenvalues of
$f_0(a)$ are positive. Together with $\det f_0(a)\neq 0$ for every $a$, the continuity of
the eigenvalues implies that no eigenvalue of $f_0(a)$ can cross $0$ as $a$ varies, hence
$f_0(a)$ is positive definite for all $a\in\R$.

Now, for each $\lambda>0$ we have
\[
f_\lambda(a) = f_0(a) + \lambda\Id.
\]
Since $f_0(a)$ is positive definite, all eigenvalues of $f_\lambda(a)$ are strictly positive
for every $(\lambda,a)\in\R^+\times\R$, and in particular $f_\lambda(a)$ is invertible for all
$a\in\R$. Lemma~\ref{lem:criteria_hypebolic} then implies that $JB_\lambda$ is hyperbolic for
all $\lambda\ge 0$.
\end{proof}

\begin{cor}\label{cor:hyper_sI}
Let $R_\lambda = R + \lambda\Id$ and assume that condition \textnormal{(L2)} holds.
Then there exists a constant $\lambda_0>0$ (for instance
\[
\lambda_0 = \frac{8C_2^2}{C_1} + C_3
\]
as in Lemma~\ref{lem:matrix positive condition}) such that, for all $\lambda\ge\lambda_0$, the
matrix $JB_\lambda$ is hyperbolic.
\end{cor}

\begin{proof}
By Lemma~\ref{lem:matrix positive condition}, for $\lambda\ge\lambda_0$ the block matrix
\[
\begin{pmatrix}
P & Q\\[2pt]
Q^T & R_\lambda
\end{pmatrix}
\]
is positive definite.  
Hence, by Corollary~\ref{cor:hyper_positive_definite}, the corresponding matrix $JB_\lambda$ is
hyperbolic.
\end{proof}

A similar statement holds if we replace the path $\lambda\mapsto R+\lambda\Id$ with the path
$\lambda\mapsto R_\lambda:=\lambda P$, which will be useful later on.

\begin{cor}\label{cor:hyper_sP}
Assume that $P$ is invertible and let $R_\lambda:=\lambda P$.  
Then there exists $\widehat\lambda>0$ such that $JB_\lambda$ is hyperbolic for every $\lambda>\widehat\lambda$.
\end{cor}

\begin{proof}
We compute
\begin{align*}
a^2P + i a(Q^T-Q) + R_\lambda
&=
(a^2+\lambda)P + i a(Q^T-Q)\\
&=
P(a^2+\lambda)\left(\Id + \frac{i a}{a^2+\lambda} P^{-1}(Q^T-Q)\right).
\end{align*}
For $\lambda>0$ we have
\[
\left|\frac{a}{a^2+\lambda}\right|
\le \frac{1}{2\sqrt\lambda},
\]
so there exists $\widehat\lambda>0$ such that, for every $\lambda>\widehat\lambda$,
\[
\left\|\frac{i a}{a^2+\lambda}P^{-1}(Q^T-Q)\right\| < 1
\quad\text{for all } a\in\R.
\]
Therefore the perturbation of the identity in the last factor is invertible, and so
$a^2P + i a(Q^T-Q) + R_\lambda$ is invertible for all $a\in\R$ and all $\lambda>\widehat\lambda$.
By Lemma~\ref{lem:criteria_hypebolic}, this implies that $JB_\lambda$ is hyperbolic for every
$\lambda>\widehat\lambda$.
\end{proof}

\medskip

The next transversality result describes the relative position of the spectral subspaces of a
hyperbolic Hamiltonian matrix with respect to the Dirichlet Lagrangian.

\begin{lem}\label{thm:transversal}
Assume that $JB_0$ is hyperbolic. Then, for all $\lambda\in\R^+$, the positive and negative spectral
subspaces of $JB_\lambda$ are both transversal to the horizontal (Dirichlet) Lagrangian, namely
\[
V^\pm(JB_\lambda)\cap L_D = \{0\}
\qquad\forall\,\lambda\in\R^+,
\]
where $L_D = \R^n\times\{0\}$.
\end{lem}

\begin{proof}
We give the argument for $V^{+}(JB_\lambda)$; the case of 
$V^{-}(JB_\lambda)$ is analogous.  
By Lemma~\ref{thm:hyperbolic-at-starting-point}, each matrix $JB_\lambda$ is 
hyperbolic for $\lambda\ge0$. Hence its positive and negative spectral 
subspaces $V^{\pm}(JB_\lambda)$ are Lagrangian in $(\R^{2n},\omega)$; 
see, for instance, \cite{HP17}.

Let $(u,0)^T\in V^+(JB_\lambda)\cap L_D$.  
Since $V^+(JB_\lambda)$ is invariant under $JB_\lambda$, we also have
$JB_\lambda(u,0)^T\in V^+(JB_\lambda)$.  
Using the symplectic form $\omega(z_1,z_2)=\langle J z_1,z_2\rangle$, we compute
\[
0
=
\omega\bigl(JB_\lambda(u,0)^T,(u,0)^T\bigr)
=
\bigl\langle J(JB_\lambda(u,0)^T),(u,0)^T\bigr\rangle
=
-\bigl\langle B_\lambda(u,0)^T,(u,0)^T\bigr\rangle.
\]
From the explicit expression of $B_\lambda$ we have
\[
B_\lambda(u,0)^T
=
\begin{pmatrix}
P^{-1}u\\
-Q^T P^{-1}u
\end{pmatrix},
\]
so that
\[
\bigl\langle B_\lambda(u,0)^T,(u,0)^T\bigr\rangle
=
\langle P^{-1}u,u\rangle.
\]
Hence
\[
0 = -\langle P^{-1}u,u\rangle.
\]
Since $P>0$, the matrix $P^{-1}$ is positive definite and therefore $\langle P^{-1}u,u\rangle=0$
implies $u=0$. Thus $(u,0)^T=0$ and
\[
V^+(JB_\lambda)\cap L_D = \{0\},
\]
as claimed.
\end{proof}

%
%
%
%

\section{First order differential operators and Fredholmness}\label{sec:fredholm-sturm}
 
In this section we collect the results about the Fredholmness of Sturm–Liouville operators both on
the line and on the half-line, as well as for the associated first order Hamiltonian operators, that
are needed in the construction of the index theory.

We will frequently pass from second order scalar equations to first order Hamiltonian systems and
back. In doing so, the same differential expression will appear with different domains (minimal,
maximal, and with boundary conditions). A recurrent technical point is to compare Fredholm
properties of such operators when their domains differ only by finite-dimensional subspaces. We
begin with a classical abstract result of this type.

\begin{lem}\label{lem:finite-codim-rge-closed}\cite[Theorem 2.4]{Kre82}
Let $X,\ Y$ be Banach spaces and let  $L:\dom L\subset X\to Y$ be a closed linear operator with
dense domain $\dom L$. Assume that there exists a closed subspace $V$ of $Y$ such that 
\[
\im L\oplus V=Y.
\]
Then $\im L$ is closed in $Y$. In particular, if $\codim \im L<+\infty$, then $\im L$ is closed in $Y$.
\end{lem}

The next lemma relates the Fredholmness of the Sturm–Liouville operator on the half-line for
different choices of domain (minimal, maximal, and with a selfadjoint boundary condition at the
initial instant).

\begin{lem}\label{thm:equivalence-SL}
With the above notation, the operator $\mathcal A^+_{L_0}$ is Fredholm if and only if the operators
$\mathcal A_m^+$ and $\mathcal A^+$ are Fredholm, where $L_0$ denotes the selfadjoint boundary
condition at the initial instant.
\end{lem}	

\begin{proof}
Recall that $\mathcal A^+$ is the maximal Sturm–Liouville operator on $W^{2,2}$, and that
$\mathcal A_m^+$ is its minimal realization on $W_0^{2,2}$. The operator $\mathcal A^+_{L_0}$ is the
realization with selfadjoint boundary condition $L_0$ at $t=0$, so we have continuous inclusions of
domains
\[
\dom \mathcal A_m^+\subset \dom \mathcal A^+_{L_0}\subset \dom \mathcal A^+,
\]
and therefore
\begin{equation}\label{eq:incl-kernel-image}
\ker\left(\mathcal  A^+_m\right)\subset \ker\left(\mathcal A^+_{L_0}\right)\subset \ker\left(\mathcal A^+_M\right), 
\qquad 
\im\left(\mathcal A^+_m\right)\subset \im\left(\mathcal A^+_{L_0}\right)\subset \im\left(\mathcal A^+_M\right),
\end{equation}
where $\mathcal A^+_M:=\mathcal A^+$ denotes the maximal operator.

Moreover, $\mathcal A^+$ and $\mathcal A_m^+$ are conjugated (in the $L^2$ sense) by a bounded
invertible operator (coming from the standard reduction to a first order Hamiltonian system). In
particular, $\mathcal A_m^+$ is Fredholm if and only if $\mathcal A^+$ is Fredholm, and in this case
they have the same index.

\medskip

\noindent$(\Leftarrow)$ Assume that $\mathcal A_m^+$ is Fredholm. Then $\ker\mathcal A_m^+$ and the
cokernel $L^2/\im\mathcal A_m^+$ are finite-dimensional, and $\im\mathcal A_m^+$ is closed. From
\eqref{eq:incl-kernel-image} we obtain
\[
\im(\mathcal A_m^+)\subset \im(\mathcal A^+_{L_0})\subset L^2,
\]
and hence
\[
\codim\bigl(\im \mathcal A^+_{L_0}\bigr)\le \codim\bigl(\im \mathcal A_m^+\bigr)<+\infty.
\]
Thus the cokernel of $\mathcal A^+_{L_0}$ is finite-dimensional. 

On the other hand, every element in $\ker\mathcal A^+_{L_0}$ (and in $\ker\mathcal A^+_M$) is a
solution of the associated second order ODE belonging to $W^{2,2}(\R^+)$, hence to $L^2(\R^+)$. 
Standard ODE theory implies that the space of such $L^2$-solutions has finite dimension (at most
$2n$ in our setting). Therefore
\[
\dim \ker\mathcal A^+_{L_0}\le \dim\ker\mathcal A^+_M<+\infty.
\]

We have shown that both the kernel and cokernel of $\mathcal A^+_{L_0}$ are finite-dimensional. To
conclude Fredholmness, it remains to prove that $\im\mathcal A^+_{L_0}$ is closed. Since
$\mathcal A^+_{L_0}$ is a closed operator and $\codim\im\mathcal A^+_{L_0}<+\infty$, we can apply
Lemma~\ref{lem:finite-codim-rge-closed} to obtain that $\im\mathcal A^+_{L_0}$ is closed in $L^2$.
Hence $\mathcal A^+_{L_0}$ is Fredholm.

\medskip

\noindent$(\Rightarrow)$ Conversely, assume that $\mathcal A^+_{L_0}$ is Fredholm. Then both its
kernel and cokernel are finite-dimensional and its range is closed. From
\eqref{eq:incl-kernel-image} we again obtain
\[
\im(\mathcal A_m^+)\subset \im(\mathcal A^+_{L_0}) \subset L^2.
\]
The quotient $\im(\mathcal A^+_{L_0})/\im(\mathcal A_m^+)$ is finite-dimensional, since passing from
$\dom\mathcal A_m^+$ to $\dom\mathcal A^+_{L_0}$ amounts to imposing a different (selfadjoint)
boundary condition at $t=0$, and the space of boundary values at $t=0$ is finite-dimensional.
Thus $\im(\mathcal A_m^+)$ has finite codimension in $L^2$ and is a finite-codimensional subspace of
the closed subspace $\im(\mathcal A^+_{L_0})$; by Lemma~\ref{lem:finite-codim-rge-closed} it follows
that $\im(\mathcal A_m^+)$ is closed. 

As above, the kernel of $\mathcal A_m^+$ consists of $L^2$-solutions of the associated second order
ODE satisfying stronger boundary conditions, and hence it is a subspace of $\ker\mathcal A^+_{L_0}$;
in particular, $\dim\ker\mathcal A_m^+<+\infty$. Therefore $\mathcal A_m^+$ is Fredholm. Since
$\mathcal A_m^+$ and $\mathcal A^+$ are conjugated, $\mathcal A^+$ is Fredholm as well.

This proves the claimed equivalence.
\end{proof}

Arguing in exactly the same way for the associated first order Hamiltonian realizations, we obtain:

\begin{lem}\label{lem:fredholm equvalent for f.l.o with dif. value}
The operator $\mathcal F^+_{L_0}$ is Fredholm if and only if the operator $\mathcal F_m^+$ is Fredholm.
\end{lem}

We now introduce a one-parameter family of first order Hamiltonian operators on the half-line that
will be used as an auxiliary tool in the comparison between the Sturm–Liouville and Hamiltonian
pictures. For each $s \in \R$, let 
\[
\mathscr F^+_{0,s}:L^2(\R^+,\R^{2n})\rightarrow L^2(\R^+,\R^{2n})
\]
be the first order differential operator defined by 
\[
\mathscr F^+_{0,s}=-J\dfrac{d }{d  t}-B_s(t)
\]
on the domain $W_0^{1,2}(\R^+,\R^{2n})$, where  
\[
B_s(t)=\begin{pmatrix}
P^{-1}(t) & -P^{-1}(t) Q(t)\\[4pt]
-Q^T(t)P^{-1}(t) & Q^T(t)P^{-1}(t) Q(t)-sP(t)
\end{pmatrix}.
\]
We denote by $B^+_{s}$ the uniform limit of $B_s(t)$ as $t \to+\infty$, whose existence is guaranteed
by our standing hypotheses.

\begin{lem}\label{lem:operator invertible}
The operator $\mathscr F^+_{0,s}$ is non-degenerate for every $s\in\R$, i.e.\ $\ker \mathscr F^+_{0,s}=\{0\}$.
\end{lem}	

\begin{proof}
Consider the associated Hamiltonian system
\begin{align}\label{eq:ode-only-0-solution}
\begin{cases}
\dot{z}(t)=JB_s(t)z(t),\qquad  t\in\R^+,\\[4pt]
z(0)=0.
\end{cases}
\end{align}
By definition, $\ker \mathscr F_{0,s}^+$ consists of all solutions $z\in W_0^{1,2}(\R^+,\R^{2n})$ of
\eqref{eq:ode-only-0-solution}. However, by the standard existence and uniqueness theorem for ODEs,
the only solution of \eqref{eq:ode-only-0-solution} with initial condition $z(0)=0$ is the trivial
solution $z\equiv 0$, for every fixed $s\in\R$. Hence $\ker \mathscr F^+_{0,s}=\{0\}$, and
$\mathscr F^+_{0,s}$ is non-degenerate for all $s\in\R$.
\end{proof}

\begin{lem}\label{lem:operator fredholm condition}
There exists $\widehat s\in\R$ such that $JB^+_s$ is hyperbolic for every  $s\ge \widehat s$.
\end{lem}	

\begin{proof}
This is a direct consequence of Corollary~\ref{cor:hyper_sP}, which provides hyperbolicity of the
limiting Hamiltonian matrices for all sufficiently large $s$.
\end{proof} 

The next result gives a characterization of the Fredholmness of the minimal Sturm–Liouville
operator $\mathcal A_m^+$ in terms of the corresponding first order Hamiltonian operator
$\mathcal F_m^+$.

\begin{prop}\label{thm:fredholm-equivalent} 
The operator
\[
\mathcal A_m^+:W_0^{2,2}(\R^+;\R^{n})\subset L^2(\R^+;\R^{n})\to L^2(\R^+;\R^{n})
\]
is Fredholm if and only if the operator $\mathcal F_m^+$ is Fredholm. Moreover,
\[
\ind \mathcal A_m^+= \ind \mathcal F_m^+.
\]
\end{prop}

\begin{proof}
We start by observing that $\mathcal F_m^+$ and $\mathcal A_m^+$ are both symmetric operators and
that their adjoints are respectively the maximal operators $\mathcal F^+$ and $\mathcal A^+$. Hence
\[
\ker(\mathcal F^+)=\im (\mathcal F_m^+)^\perp ,
\qquad
\ker(\mathcal A^+)=\im (\mathcal A_m^+)^\perp.
\]
Moreover, it is well-known that 
\[
\dim\ker(\mathcal A^+)=\dim \ker(\mathcal F^+)\le 2n,
\qquad 
\dim\ker(\mathcal A_m^+)=\dim \ker(\mathcal F_m^+)=0,
\]
the last equalities following from the fact that $\mathcal A_m^+$ and $\mathcal F_m^+$ incorporate
homogeneous boundary conditions at $t=0$ and at $+\infty$. Thus, to conclude the equivalence of
Fredholmness it remains to show that
\begin{itemize}
\item $\im(\mathcal A_m^+)$ is closed  if and only if  $\im(\mathcal F_m^+)$ is closed.
\end{itemize}

To this end, let us consider the closed subspaces  
\[
H_1=\left\{\begin{pmatrix}v\\ 0\end{pmatrix}\,\middle|\,v\in L^2(\R^+,\R^n)\right\},
\qquad
H_2=\left\{\begin{pmatrix}0\\ u\end{pmatrix}\,\middle|\,u\in L^2(\R^+,\R^n)\right\},
\]
so that $L^2(\R^+,\R^{2n})=H_1\oplus H_2$.

\paragraph{First claim.} We show that
\[
\im (\mathcal F_m^+) \text{ closed }\quad \Rightarrow \quad \im (\mathcal A_m^+) \text{ closed.}
\]
A straightforward computation gives
\begin{align}\label{eq:equivalent-fredholm-2}
\mathcal F_m^+\begin{pmatrix}y\\ x\end{pmatrix} =\begin{pmatrix}0\\ h\end{pmatrix}
\quad\Longleftrightarrow\quad  
y=P\dot x +Qx
\text{\ and\ } \mathcal A^+_m x=h.
\end{align}
Thus, \eqref{eq:equivalent-fredholm-2} implies  
\begin{align}\label{eq:rge A=rge F 2}
h\in \im(\mathcal A_m^+)\Longleftrightarrow
\begin{pmatrix}0\\ h\end{pmatrix}\in \im(\mathcal F_m^+).
\end{align} 
In other words, the subspace $H_2\cap \im (\mathcal F_m^+)$ is naturally isomorphic to
$\im (\mathcal A^+_m)$ via the identification
\[
H_2\ni \begin{pmatrix}0\\ h\end{pmatrix} \longleftrightarrow h\in L^2(\R^+,\R^n).
\]
If $\im (\mathcal F_m^+)$ is closed in $L^2(\R^+,\R^{2n})$, then
$H_2\cap \im (\mathcal F_m^+)$ is closed in $H_2$, hence also in $L^2(\R^+,\R^{2n})$. By the above
identification, $\im (\mathcal A_m^+)$ is then closed in $L^2(\R^+,\R^n)$.

\paragraph{Second claim.} We now prove the converse implication
\[
\im (\mathcal A_m^+) \text{ closed } \quad \Rightarrow \quad\im (\mathcal F_m^+) \text{ closed.}
\]

Assume that $\im (\mathcal A_m^+)$ is closed. To conclude, it is enough to show that
$H_2\cap \im \mathcal F_m^+$ is closed in $L^2(\R^+,\R^{2n})$, since $\im\mathcal F_m^+$ is then a
finite-codimensional extension of this closed subspace (as we will see below).

Let $s>\widehat{s}$, where $\widehat{s}$ is given by Lemma~\ref{lem:operator fredholm condition}. By
Lemmas~\ref{lem:operator invertible} and~\ref{lem:operator fredholm condition}, the operator
$\mathcal F^+_{0,s}$ is Fredholm with trivial kernel, hence invertible onto its image. By the closed
graph theorem, the inverse 
\[
\left(\mathcal F^+_{0,s}\right)^{-1}:\im \mathcal F^+_{0,s}\to W_0^{1,2}(\R^+,\R^{2n})
\]
is bounded.

Let $f\in \im \mathcal F_{0,s}^+$. We compute
\begin{align}\label{eq:equivalent of closed operator}
f-\mathcal F_m^+ \left(\mathcal F_{0,s}^+\right)^{-1}f
&=\left(\mathcal F^+_{0,s}-\mathcal F_m^+\right)\left(\mathcal F_{0,s}^+\right)^{-1}f\\
&=
\begin{pmatrix}0 & 0\\[2pt]
0 & R_+-sP_+
\end{pmatrix}
\left(\mathcal F_{0,s}^+\right)^{-1}f \in  H_2.\nonumber
\end{align}
Define
\[
T=\Id-\mathcal F_m^+\left(\mathcal F_{0,s}^+\right)^{-1}:\im\mathcal F_{0,s}^+\to H_2.
\]
Then $T$ is continuous and by \eqref{eq:equivalent of closed operator} we have
\[
Tf \in \im \mathcal F_m^+\quad\Longleftrightarrow\quad f\in \im \mathcal F_m^+.
\]
Consequently,
\begin{equation}
\im \mathscr F^+_{0,s}\cap \im \mathscr F^+_m
=
T^{-1}\bigl(H_2\cap \im \mathscr F^+_m\bigr).
\end{equation}
Since $T$ is continuous, the set $H_2\cap \im \mathscr F^+_m$ is closed if and only if
$\im \mathscr F^+_{0,s}\cap \im \mathscr F^+_m$ is closed in $\im\mathscr F^+_{0,s}$.

On the other hand, $\im \mathcal F^+_m$ is a finite-codimensional extension of
$\im \mathcal F^+_{0,s}\cap \im \mathcal F^+_m$ in $\im \mathcal F^+_m$. Indeed, writing
$X= L^2(\R^+,\R^{2n})$, we have
\begin{equation}
\im \mathcal F^+_m/(\im \mathcal F^+_{0,s}\cap \im \mathcal F^+_m)\cong (\im \mathcal F^+_{0,s}+ \im \mathcal F^+_m)/\im \mathcal F^+_{0,s},
\end{equation}	
and the latter quotient has dimension
\[
\dim \bigl((\im \mathcal F^+_{0,s}+ \im \mathcal F^+_m)/\im \mathcal F^+_{0,s}\bigr)
\le \codim \im \mathcal F_{0,s}^+<+\infty,
\]
because $\mathcal F^+_{0,s}$ is Fredholm. Hence
\[
\dim\bigl(\im \mathcal F^+_m/(\im \mathcal F^+_{0,s}\cap \im \mathcal F^+_m)\bigr)<\infty.
\]
Therefore $\im \mathcal F^+_m$ is the direct sum of the closed subspace
$\im \mathcal F^+_{0,s}\cap \im \mathcal F^+_m$ and a finite-dimensional space, and thus is closed.

Combining the two claims, we obtain the equivalence of the Fredholm property for $\mathcal A_m^+$
and $\mathcal F_m^+$. The identity of the indices follows from the kernel and cokernel
identifications above.
\end{proof} 
		
We now turn to constant-coefficient Hamiltonian operators. Let
\[
B(+\infty)=\begin{pmatrix}
P^{-1}_+&-P^{-1}_+Q_+\\[4pt]
-Q^T_+P^{-1}_+&Q^T_+P^{-1}_+Q_+-R_+
\end{pmatrix},
\]
where $P_+, Q_+, R_+$ are the matrices appearing in condition \textnormal{(H2)}. We define the
operators 
\begin{align}\label{eq:constant hamiltonian operator+}
\mathcal F_{L_0}^{+\infty}
&=
-J\frac{d }{d  t}-B(+\infty):W^+_{L_0}(\R^+,\R^{2n})\subset L^2\left(\R^+,\R^{2n}\right)\to L^2\left(\R^+,\R^{2n}\right),\nonumber \\
\mathcal F_m^{+\infty}
&=
-J\frac{d }{d  t}-B(+\infty): W_0^{1,2}(\R^+,\R^{2n})\subset L^2\left(\R^+,\R^{2n}\right)\to L^2\left(\R^+,\R^{2n}\right),\nonumber \\
\mathcal F^{+\infty}
&=
-J\frac{d }{d  t}-B(+\infty):W^{1,2}(\R^+,\R^{2n}) \subset L^2\left(\R^+,\R^{2n}\right)\to L^2\left(\R^+,\R^{2n}\right).
\end{align}	

The next lemma shows that the variable-coefficient Hamiltonian operator $\mathcal F^+$ is a
relatively compact perturbation of its constant-coefficient limit $\mathcal F_{L_0}^{+\infty}$.

\begin{lem}\label{lem:relative compact perturbation}
The operator
\[
\mathcal F^+=-J\frac{d }{d  t}-B(t):W^+_{L_0}(\R^+,\R^{2n})\subset L^2\left(\R^+,\R^{2n}\right)\to L^2\left(\R^+,\R^{2n}\right)
\] 
is a relatively compact perturbation of the operator $\mathcal F_{L_0}^{+\infty}$ given in
\eqref{eq:constant hamiltonian operator+}.
\end{lem}

\begin{proof} 
Fix $\lambda$ in the resolvent set of $\mathcal F_{L_0}^{+\infty}$. We need to show that the operator 
\[
L_{\lambda}= \bigl(\mathcal F_{L_0}^+-\mathcal F_{L_0}^{+\infty}\bigr)\circ\bigl(\mathcal F_{L_0}^{+\infty}-\lambda\Id\bigr)^{-1}:L^2\left(\R^+,\R^{2n}\right)\to L^2\left(\R^+,\R^{2n}\right)
\] 
is compact.

Let $\{\chi_k\}_{k \in \N}$ be a sequence of $\mathscr C^\infty$ cut-off functions on $\R^+$ such
that $\sup_{t\in\R^+}|\chi_k(t)|\le 1$,  $\sup_{t\in\R^+}|\chi'_k(t)|\le 1$ and 
\[
\chi_k(t)=\begin{cases}
1 & \textrm{ if }t\in[0,k-1],\\[4pt]
0 & \textrm{ if }t\in[k,+\infty).
\end{cases}
\]
Define the bounded multiplication operator 
\[
\alpha_k:L^2(\R^+,\R^{2n})\to L^2(\R^+,\R^{2n}),\qquad 
\alpha_k(x)(t)=\chi_k(t)\,x(t),
\]
and consider
\[
L_{k,\lambda}= \bigl(\mathcal F_{L_0}^+-\mathcal F_{L_0}^{+\infty}\bigr)\circ\alpha_k\circ\bigl(\mathcal F_{L_0}^{+\infty}-\lambda\Id\bigr)^{-1}:L^2\left(\R^+,\R^{2n}\right)\to L^2\left(\R^+,\R^{2n}\right).
\] 
We have
\[
\Bigl[\bigl(\mathcal F_{L_0}^+-\mathcal F_{L_0}^{+\infty}\bigr)(\alpha_k-\Id)\, x \Bigr](t)
=
\bigl(B(+\infty) -B(t)\bigr)\bigl(\chi_k(t)-1\bigr)x(t),
\] 
and by the assumption that $B(t)\to B(+\infty)$ uniformly as $t\to+\infty$ we obtain
\[
\lim_{k\to \infty}\sup_{t\in\R^+}\|B(+\infty)-B(t)\|\cdot|\chi_k(t)-1|=0.
\]
Thus $\bigl(\mathcal F_{L_0}^+-\mathcal F_{L_0}^{+\infty}\bigr)(\alpha_k-\Id)$ converges to $0$ in the
operator norm as $k\to+\infty$, and hence $L_{k,\lambda}\to L_\lambda$ in operator norm. Since the
set of compact operators is norm-closed, it suffices to show that each $L_{k,\lambda}$ is compact.

Now
\[
\bigl(\mathcal F_{L_0}^{+\infty}-\lambda\Id\bigr)^{-1}:L^2\left(\R^+,\R^{2n}\right)\to   W^{1,2}(\R^+,\R^{2n})
\]
is bounded. By construction, $\chi_k(t)=0$ for all $t\ge k$ and
$\sup_{t\in\R^+}|\chi_k(t)|\le 1$, so that 
\[
\alpha_k :W^{1,2}(\R^+,\R^{2n})\to W^{1,2}([0,k],\R^{2n})
\]
is bounded. By the compact Sobolev embedding, $W^{1,2}([0,k],\R^{2n})$ is compactly embedded in
$L^2([0,k],\R^{2n})$. Thus
\[
\alpha_k \circ \bigl(\mathcal F_{L_0}^{+\infty}-\lambda\Id\bigr)^{-1}:L^2(\R^+,\R^{2n})\to L^2(\R^+,\R^{2n})
\]
is a compact operator. Since $\mathcal F_{L_0}^+-\mathcal F_{L_0}^{+\infty}$ is bounded on $L^2$,
the composition $L_{k,\lambda}$ is compact. This proves that $L_\lambda$ is compact, and hence that
$\mathcal F^+$ is a relatively compact perturbation of $\mathcal F_{L_0}^{+\infty}$.
\end{proof}

We now characterize the Fredholmness of the constant-coefficient operator $\mathcal F^{+\infty}$ in
terms of the spectral properties of the matrix $JB(+\infty)$.

\begin{lem}\label{lem:constant fredholm iif hyperbolic}
The operator $\mathcal F^{+\infty}$ defined in \eqref{eq:constant hamiltonian operator+} is Fredholm
if and only if the matrix $JB(+\infty)$ is hyperbolic. In this case, its Fredholm index equals the
dimension of the negative spectral space of $JB(+\infty)$, namely
\[
\ind \mathcal F^{+\infty}=\dim V^-(JB(+\infty)).
\]
\end{lem}

\begin{proof}
By  \cite[Theorem 2.3]{RS05b}, the operator
\[
\mathcal G^{+\infty}=\dfrac{d }{d  t}-JB(+\infty),
\]
defined on $W^{1,2}\left(\R^+,\R^{2n}\right)$, is Fredholm if and only if $JB(+\infty)$ is
hyperbolic. Moreover,
\[
\ind \mathcal G^{+\infty}=\dim V^-(JB(+\infty)).
\]

It is straightforward to check that $\mathcal F^{+\infty}=-J\mathcal G^{+\infty}$, so that
\[
\im \mathcal F^{+\infty}=-J\im\mathcal G^{+\infty}.
\]
Since $-J$ is an isomorphism of $L^2(\R^+,\R^{2n})$, it follows that 
\[
\codim \im \mathcal F^{+\infty}=\codim \im \mathcal G^{+\infty},
\qquad
\ker\mathcal F^{+\infty}=\ker\mathcal G^{+\infty}.
\]
Thus $\mathcal F^{+\infty}$ is Fredholm if and only if $\mathcal G^{+\infty}$ is Fredholm, and in
this case
\[
\ind \mathcal F^{+\infty}=\ind \mathcal G^{+\infty}=\dim V^-(JB(+\infty)).
\]
The closedness of the ranges of these operators under the above hypotheses follows from
Lemma~\ref{lem:finite-codim-rge-closed} and \cite[Lemma 2.1]{RS05b}.
\end{proof} 
 
\begin{cor}\label{lem:min s.l.o. fredhom iff hyperbolic}
The operator $\mathcal F_m^{+\infty}$ defined in \eqref{eq:constant hamiltonian operator+} is
Fredholm if and only if the matrix $JB(+\infty)$ is hyperbolic. 
\end{cor}

\begin{proof}
The operators $\mathcal F_m^{+\infty}$ and $\mathcal F^{+\infty}$ are conjugated by a bounded
invertible operator (corresponding to the choice of boundary condition at the origin). Hence they
are simultaneously Fredholm, with the same index. The claim is therefore an immediate consequence of
Lemma~\ref{lem:constant fredholm iif hyperbolic}.
\end{proof}

\vspace{0.5cm}

\noindent
\begin{minipage}[t]{0.47\linewidth}
\textsc{Xijun Hu}\\
School of Mathematics\\
Shandong University\\
State Key Laboratory of Cryptography\\
and Digital Economy Security\\
Jinan, Shandong 250100\\
People's Republic of China\\
E-mail: \texttt{xjhu@sdu.edu.cn}

\vspace{0.6cm}

\textsc{Li Wu}\\
School of Mathematics\\
Shandong University\\
Jinan, Shandong 250100\\
People's Republic of China\\
E-mail: \texttt{201790000005@sdu.edu.cn}
\end{minipage}
\hfill
\begin{minipage}[t]{0.47\linewidth}
\textsc{Alessandro Portaluri}\\
Università degli Studi di Torino (DISAFA)\\
Largo Paolo Braccini 2\\
10095 Grugliasco (TO), Italy\\
Website: \url{https://portalurialessandro.wordpress.com}\\
E-mail: \texttt{alessandro.portaluri@unito.it}\\[2mm]
Visiting Professor of Mathematics\\
NYU Abu Dhabi, UAE\\
E-mail: \texttt{ap9453@nyu.edu}

\vspace{0.6cm}

\textsc{Qin Xing}\\
School of Mathematics and Statistics\\
Linyi University\\
Linyi, Shandong 276000\\
People's Republic of China\\
E-mail: \texttt{xingqin@lyu.edu.cn}
\end{minipage}

\vspace{1cm}


\begin{thebibliography}{10}

\bibitem[AM03]{AM03}
{\sc Abbondandolo, Alberto; Majer, Pietro}
\newblock Ordinary differential operators in Hilbert spaces and Fredholm pairs. 
\newblock Math. Z. 243 (2003), no. 3, 525--562.

\bibitem[BY11]{BY11}
{\sc Bao, Jianghong; Yang, Qigui}
\newblock A new method to find homoclinic and heteroclinic orbits.
\newblock Appl. Math. Comput. 217 (2011), no. 14, 6526--6540.

\bibitem[BHPT19]{BHPT19}
{\sc Barutello, Vivina; Hu, Xijun; Portaluri, Alessandro; Terracini, Susanna}
\newblock An index theory for asymptotic motions under singular potentials.
\newblock Adv. Math. 370 (2020), 107230.

\bibitem[BLP05]{BLP05}
{\sc Booss-Bavnbek, Bernhelm; Lesch, Matthias; Phillips, John}
\newblock Unbounded Fredholm operators and spectral flow. 
\newblock Can. J. Math. 57 (2005), no. 2, 225--250. 

\bibitem[CLM94]{CLM94}
{\sc Cappell, Sylvain E.; Lee, Ronnie; Miller, Edward Y.}
\newblock On the Maslov index.
\newblock Comm. Pure Appl. Math. 47 (1994), no. 2, 121--186.

\bibitem[CB15]{CB15}
{\sc Chardard, Frédéric; Bridges, Thomas J.}
\newblock Transversality of homoclinic orbits, the Maslov index and the symplectic Evans function.
\newblock Nonlinearity 28 (2015), no. 1, 77--102.

\bibitem[CDB09a]{CDB09a}
{\sc Chardard, Frédéric; Dias, Frédéric; Bridges, Thomas J.}
\newblock Computing the Maslov index of solitary waves. I. Hamiltonian systems on a four-dimensional phase space.
\newblock Phys. D 238 (2009), no. 18, 1841--1867.

\bibitem[CDB09b]{CDB09b}
{\sc Chardard, Frédéric; Dias, Frédéric; Bridges, Thomas J.}
\newblock On the Maslov index of multi-pulse homoclinic orbits.
\newblock Proc. R. Soc. Lond. Ser. A Math. Phys. Eng. Sci. 465 (2009), no. 2109, 2897--2910.

\bibitem[CDB11]{CDB11}
{\sc Chardard, Frédéric; Dias, Frédéric; Bridges, Thomas J.}
\newblock Computing the Maslov index of solitary waves. II. Phase space with dimension greater than four.
\newblock Phys. D 240 (2011), no. 17, 1334--1344.

\bibitem[CH07]{CH07} 
{\sc Chen, Chao-Nien; Hu, Xijun}
\newblock Maslov index for homoclinic orbits of Hamiltonian systems.
\newblock Ann. Inst. H. Poincaré Anal. Non Linéaire 24
(2007), no. 4, 589--603.

\bibitem[JM12]{JM12}
{\sc Jones, Christopher K. R. T.; Marangell, Robert}
\newblock The spectrum of travelling wave solutions to the sine-Gordon equation.
\newblock Discrete Contin. Dyn. Syst. Ser. S 5 (2012), no. 5, 925--937.

\bibitem[HLS17]{HLS17}
{\sc Howard, P.; Latushkin, Y.; Sukhtayev, A.}
\newblock The Maslov index for Lagrangian pairs on $\R^{2n}$. 
\newblock J. Math. Anal. Appl. 451 (2017), no. 2, 794--821.

\bibitem[HLS18]{HLS18}
{\sc Howard, Peter; Latushkin, Yuri; Sukhtayev, Alim}
\newblock The Maslov and Morse indices for system Schr\"odinger operators on $\R$.
\newblock Indiana Univ. Math. J. 67 (2018), no. 5, 1765--1815.

\bibitem[HS20]{HS20}
{\sc Howard, Peter; Sukhtayev, Alim}
\newblock The Maslov and Morse indices for Sturm--Liouville systems on the half-line.
\newblock Discrete Contin. Dyn. Syst. 40 (2020), no. 2, 983--1012.

\bibitem[How21]{How21}
{\sc Howard, Peter}
\newblock H\"ormander's index and oscillation theory.
\newblock J. Math. Anal. Appl. 500 (2021), no. 1, Paper No.~125076, 38 pp.

\bibitem[How23]{How23}
{\sc Howard, Peter}
\newblock The Maslov index and spectral counts for linear Hamiltonian systems on $\R$.
\newblock J. Dynam. Differential Equations 35 (2023), no. 3, 1947--1991.

\bibitem[How25]{How25}
{\sc Howard, Peter}
\newblock Oscillation theory and instability of nonlinear waves.
\newblock Discrete Contin. Dyn. Syst. 45 (2025), no. 8, 2789--2855.

\bibitem[HP17]{HP17}
{\sc Hu, Xijun; Portaluri, Alessandro}
\newblock Index theory for heteroclinic orbits of Hamiltonian systems. 
\newblock Calc. Var. Partial Differential Equations 56 (2017), no. 6, Art. 167.
\newblock Preprint available at \url{https://arxiv.org/abs/1703.03908}.

\bibitem[HPWX20]{HPWX20}
{\sc Hu, X.; Portaluri, A.; Wu, L.; Xing, Q.}
\newblock Morse index for heteroclinic orbits of Lagrangian systems.
\newblock Preprint (version v1) available at \url{https://arxiv.org/pdf/2004.08643v1}.

\bibitem[HW20]{HW18}
{\sc Hu, Xijun; Wu, Li}
\newblock Decomposition of spectral flow and Bott-type iteration formula.
\newblock Electron. Res. Arch. 28 (2020), no. 1, 127--148. 

\bibitem[HWY18]{HWY18}
{\sc Hu, Xijun; Wu, Li; Yang, Ran}
\newblock Morse index theorem of Lagrangian systems and stability of brake orbits.
\newblock J. Dynam. Differential Equations (2018), \url{https://doi.org/10.1007/s10884-018-9711-x}.

\bibitem[IF06]{IF06}
{\sc Izhikevich, E. M.; FitzHugh, R.}
\newblock FitzHugh--Nagumo model. 
\newblock Scholarpedia 1 (2006), no. 9, 1349.

\bibitem[Kat80]{Kat80} 
{\sc Kato, Tosio}
\newblock Perturbation Theory for Linear Operators.
\newblock Springer, New York, 1980.

\bibitem[KT83]{KT83}
{\sc Kawahara, T.; Tanaka, K.}
\newblock Interaction of traveling fronts: an exact solution of a nonlinear diffusion equation. 
\newblock Phys. Lett. A 97 (1983), no. 8, 311--314.

\bibitem[Kre82]{Kre82}
{\sc Krein, S. G.}
\newblock Linear Equations in Banach Spaces.
\newblock Birkhäuser, Boston, 1982.

\bibitem[M78]{M78}
{\sc McKean Jr, H. P.}
\newblock Nagumo's equation.
\newblock Adv. Math. 4 (1970), no. 3, 209--223.

\bibitem[Pej08]{Pej08}
{\sc Pejsachowicz, Jacobo}
\newblock Bifurcation of homoclinics of Hamiltonian systems.
\newblock Proc. Amer. Math. Soc. 136 (2008), no. 6, 2055--2065.

\bibitem[RS05a]{RS05a}
{\sc Rabier, Patrick J.; Stuart, Charles A.}
\newblock A Sobolev space approach to boundary value problems on the half-line. 
\newblock Commun. Contemp. Math. 7 (2005), no. 1, 1--36. 

\bibitem[RS05b]{RS05b}
{\sc Rabier, Patrick J.; Stuart, Charles A.}
\newblock Boundary value problems for first order systems on the half-line. 
\newblock Topol. Methods Nonlinear Anal. 25 (2005), no. 1, 101--133.

\bibitem[RS93]{RS93}
{\sc Robbin, Joel; Salamon, Dietmar}
\newblock The Maslov index for paths. 
\newblock Topology 32 (1993), no. 4, 827--844.

\bibitem[RS95]{RS95}
{\sc Robbin, Joel; Salamon, Dietmar}
\newblock The spectral flow and the Maslov index.
\newblock Bull. London Math. Soc. 27 (1995), no. 1, 1--33.

\bibitem[Waa15]{Waa15}
{\sc Waterstraat, Nils}
\newblock Spectral flow, crossing forms and homoclinics of Hamiltonian systems.
\newblock Proc. Lond. Math. Soc. (3) 111 (2015), no. 2, 275--304.

\bibitem[Waa21]{Waa21}
{\sc Waterstraat, Nils}
\newblock On the Fredholm Lagrangian Grassmannian, spectral flow and ODEs in Hilbert spaces.
\newblock J. Differential Equations 303 (2021), 667--700.

\bibitem[ZWZ18]{ZWZ18}
{\sc Zhou, Y.; Wu, L.; Zhu, C.}
\newblock H\"ormander index in finite-dimensional case.
\newblock Front. Math. China 13 (2018), no. 3, 725--761.

\end{thebibliography}
\end{document}